\documentclass[sn-mathphys,Numbered]{sn-jnl}
\usepackage{graphics}
\usepackage{graphicx}  
  \usepackage{epsfig} 

\usepackage{amsmath,amssymb,amsfonts}%
\usepackage{amsthm}%
\usepackage{mathrsfs}%
\usepackage[title]{appendix}%
\usepackage{xcolor}%
\usepackage{textcomp}%
\usepackage{manyfoot}%
\usepackage{float}
\usepackage{booktabs}%
\usepackage{algorithm}%
\usepackage{algorithmicx}%
\usepackage{algpseudocode}%
\usepackage{listings}%
\newtheorem{theorem}{Theorem}
\newtheorem{lemma}{Lemma}
\newtheorem{corollary}{Corollary}
\newtheorem{conjecture}{Conjecture}
\newtheorem{proposition}[theorem]{Proposition}%
\DeclareMathOperator{\tr}{tr}

\newtheorem{Remark}{Remark}
\newtheorem{example}{Example}%
\newtheorem{definition}{Definition}%

\begin{document}

\title{Exploring unique dynamics in a predator-prey model with generalist predator and group defence in prey}

\author[1]{\fnm{Vaibhava} \sur{Srivastava}}\email{vaibhava@iastate.edu}
\equalcont{These authors contributed equally to this work.}

\author*[2]{\fnm{Kwadwo } \sur{Antwi-Fordjour}}\email{kantwifo@samford.edu}
\equalcont{These authors contributed equally to this work.}

\author[1]{\fnm{Rana D.} \sur{ Parshad}}\email{rparshad@iastate.edu}
\equalcont{These authors contributed equally to this work.}

\affil*[1]{\orgdiv{Department of Mathematics}, \orgname{Iowa State University}, \orgaddress{\street{Ames}, \postcode{50011}, \state{IA}, \country{USA}}}

\affil[2]{\orgdiv{Department of Mathematics and Computer Science}, \orgname{Samford University}, \orgaddress{ \city{Birmingham}, \postcode{35229}, \state{AL}, \country{USA}}}


\abstract{In the current manuscript we consider a predator-prey model where the predator is modeled as a generalist using a modified Leslie-Gower scheme, and the prey exhibits group defence via a generalised response. We show that the model could exhibit finite time blow-up, contrary to the current literature (Eur. Phys. J. Plus 137, 28). 
We also propose a new concept via which the predator population blows up in finite time while the prey population \emph{quenches} in finite time. The blow-up and quenching times are proved to be one and the same. Our analysis is complemented by numerical findings. This includes a numerical description of the basin of attraction for large data blow-up solutions, as well as several rich bifurcations leading to multiple limit cycles, both in co-dimension one and two. 
Lastly, we posit a delayed version of the model with globally existing solutions for any initial data.}


\maketitle

\section{Introduction}\label{sec1}

The current work aims to comment on the dynamics of the predator-prey model introduced in \cite{patra2022effect}. This model takes the following form,

\begin{equation}
	\left\{\begin{array}{l}
		\displaystyle\frac{d X}{d t}=R X\left(1-\displaystyle\frac{X(t-\tau)}{K}\right) - \dfrac{M X Y}{X^p + C}\vspace{2ex}\\
		\displaystyle\frac{d Y}{d t}=\left( D- \dfrac{E}{X + A} \right) Y^2
	\end{array}\right.
	\label{model}
\end{equation}
This system of equations represents the dynamics of an interacting predator and prey species, where $X(t)$ and $Y(t)$ are a prey and predator population, respectively. The prey population obeys logistic growth, with a delay term representing gestation effect \cite{Kuang93}. The prey population's functional response is non-standard, taking the form $\frac{M X }{X^p + C}$, where $p\geq 1$. If $p=1$, this is the classical Holling type II response. If $p = 2$, it is the Holling type IV response. If $1<p<2$, it is called the Cosner response \cite{Cosner99}. The predator population is modeled according to the modified Leslie-Gower scheme \cite{Upa-Iye-Rai}. The other parameters used in the model have the following interpretations: $R$ and $K$ are the intrinsic growth rate and carrying capacity of the prey population, $M$ is the maximum predation rate, $C$ represents the protection provided to the prey by the environment, $D$ is the reproduction rate of the generalist predator by sexual reproduction, $E$ is the maximum rate of death of predator population, and $A$ represents the residual loss of the predator population due to the scarcity of prey species $X$.

 In the event that there is no delay, or $\tau = 0$, the model is a subsystem of the 
well-known Upadhyay-Rai food chain model that models the interaction between three trophic groups, a prey species $x$ is depredated on by a middle predator species $y$ that in turn is depredated on by a top predator species $z$, \cite{Upa-Iye-Rai}. The top predator ($Y(t)$ here) is assumed to be a generalist. That is, it can switch its food source in the absence of its favorite food $X$. The growth of this top predator $Y$ is modeled in a nonstandard way via the modified Leslie--Gower scheme. Here the $Y$ population grows due to sexual reproduction - essentially modeled as a direct product between the males and females in the population, $\frac{Y}{2} \times \frac{Y}{2} = \frac{1}{4}Y^{2}$, which is controlled by a density-dependent, growth rate $\left( D- \frac{E}{X + A} \right) $. This term is positive if the food source, $X$, is above a certain threshold, and in this case, the top predator population grows, and if below this threshold, it decays. 

The Upadhyay-Rai model has a long-standing history in the mathematical ecology literature. It is the first model to show chaotic dynamics when the top predator is a generalist - and the second after the Hastings-Powell model \cite{HP91}, to show chaos in a three-species food chain model. Boundedness for the model was first shown in \cite{Ala}. However, in 2015, the Upadhyay-Rai model was shown to possess solutions that blow up in finite time for sufficiently large initial data \cite{Par-Kum-Kou}, and then subsequently even for small initial data \cite{Par-Qua-Bea-Kou}. Since then, many works have appeared in the literature that has attempted to either dampen the blow-up dynamic present in the basic model via  ecological mechanisms such as time delay, fear effects, or refuge \cite{PQB16, VAP21}, while others have attempted to show that blow-up does not occur altogether, via mechanisms, such as toxins, Allee effects, and fractional derivatives \cite{MR16, DS20}.

The following results concerning the properties of system \eqref{model} are proved in \cite{patra2022effect}. 

\begin{lemma}[Lemma 3.1 \cite{patra2022effect}]
\label{lem:bdd_patra}
The positive quadrant $\mathbb{R}^{2}$ is invariant for system \eqref{model}.
\end{lemma}

\vspace{2mm}

\begin{theorem}[Theorem 3.2 \cite{patra2022effect}]\label{thm:bdd_patra}
    The solution of system \eqref{model}, originating in $\mathbb{R}^2_{+}$ are bounded, provided the conditions
    \begin{align}\label{res}
        \mu < \dfrac{D X^*}{(X^{*})^p + C} \quad \& \quad D- \dfrac{E}{X^* + A} < \dfrac{D}{M}\Big( \dfrac{X^*}{Y^*}\Big)^2
    \end{align}
    holds, where $(X^*,Y^*)$ is the interior equilibrium point and $\mu=\min \{ M,E\}.$
\end{theorem}

In the current manuscript, we show the following:

\begin{itemize}
\item Solutions to \eqref{model} are unbounded. In fact, they can blow up in finite time for sufficiently large data, even under the parametric restrictions of Theorem 3.2 \cite{patra2022effect}. This is shown via Theorem \ref{thm:thm2}.

\item The positive quadrant $\mathbb{R}^{2}$ is not invariant for system \eqref{model}. This is shown via Corollary \ref{cor:bdd_patra11}.

\item Introducing delay (motivated by gestation period) in the predator yields global in-time existence. This is shown via Theorem \ref{delay_ode_blowup}. 

\item A new concept of quenching in finite time is introduced for the prey species. Subsequently, we proved that the predator's blow-up time and the prey's quenching time are one and the same. For the non-delayed model, this is shown via Theorem \ref{conj:quench1} and numerically validated through Remark \ref{rem:quen1} and Figure \ref{fig:quench}. Moreover, for the delayed model, it is proposed via Theorem \ref{conj:quench2} and numerically validated through Remark \ref{rem:quen2} and Figure \ref{fig:quench_delay}.

\item Several numerical simulations are performed to corroborate all of our analytical findings. See Figs. [\ref{fig:blowup},\ref{fig:blowup_ic},\ref{fig:quench},\ref{fig:blowup_delay},\ref{fig:quench_delay},\ref{fig:blowup_non_delay_f1},\ref{fig:blowup_non_delay_f1_domain},\ref{fig:blowup_delay_f1},\ref{fig:blowup_phi}].

\item It is numerically validated that the linear feedback control proposed in \cite{patra2022effect} does not prevent the blow-up for both delayed and non-delayed models. This is numerically shown via Remarks \ref{rem:feed1} and \ref{rem:feed2}, and Figs [\ref{fig:blowup_non_delay_f1},\ref{fig:blowup_delay_f1}].

\item In \cite{patra2022effect}, the asymptotic stability of the interior equilibrium for the delayed linear feedback control model (under some parametric restrictions) is proposed and proved. This is disproved by providing a counterexample. This is numerically shown via Remark \ref{rem:feed2} and Figure \ref{fig:blowup_delay_f1}.

\item The basin of attraction is thoroughly discussed, supported by several theorems (refer to Theorems \ref{thm:stable_int} and \ref{thm:hopf-bif}) and illustrated through Example \ref{exam1}.

\item In \cite{patra2022effect}, the authors observed the presence of multiple concurrent limit cycles. Expanding on their findings, we propose several conjectures (refer to Conjectures \ref{con3}, \ref{con4}, and \ref{con5}) and provide numerical validations (refer to Remark \ref{remark:LPC}, and Figures \ref{fig:2limit-cycles} and \ref{fig:generalized-hopf}) regarding the occurrence of two concurrent limit cycles. These validations are obtained by systematically varying either one or two parameters.

\item A detailed discussion is carried out with applications of our results to ecological scenarios. Furthermore, several open questions are proposed and discussed.

\end{itemize}

\section{Preliminaries}

Given a system of ODE, depending on the non-linearity in the equations,
one might not expect a solution to always exist globally in time. In
particular, solutions of some ODE may blow up in finite time \cite{Qui-Soup}. Recall,

\begin{definition}
(Finite time blow-up for ODE) We say that a solution $u(t)$, of a given ODE, with
suitable initial conditions, blows up at a finite time if%
\begin{equation*}
\underset{t \rightarrow T^{\ast } < \infty}{\lim }\left\vert u\left( t\right)
\right\vert =+\infty ,
\end{equation*}%
where $T^{\ast } < \infty$ is the blow-up time.
\end{definition}

\bigskip We need the following alternative (See \cite{Fri,Hen,Paz,Smo,Rot}).

\begin{proposition}
\label{prop.exis.blow.2.2} The two-component system \eqref{model}
admits a unique local in time, classical solution $\left( X, Y \right) $
on an interval $[0,T_{\max }]$, and either\newline
(i) \textit{The solution} \textit{is bounded on} $[0,T_{\max })$ \textit{and
it is global ( i.e.} $T_{\max }=+\infty $).\newline
\textit{(ii) Or }%
\begin{equation}
\underset{t\nearrow T_{\max }}{\lim }\max \hspace{.02in} \left\vert X(t)\right\vert
+\left\vert Y(t)\right\vert  =+\infty ,
\label{2.D.C.R.2.3}
\end{equation}%
\textit{in this case, the solution is not global, and we say that it blows up in finite
time} $T_{\max }$, \textit{or it ceases to exist}, where $T_{\max } < \infty$
denotes the eventual blowing-up time.
\end{proposition}

\begin{definition}
(Finite time quenching for ODE) We say that a solution $u(t)$ of a given ODE, with
suitable initial conditions, quenches in finite time if%
\begin{equation*}
\underset{t \rightarrow T^{\ast \ast} < \infty}{\lim }\left\vert \frac{du\left( t\right)}{dt}
\right\vert =+\infty ,
\end{equation*}%
where $T^{\ast \ast} < \infty$ is the quenching time.
\end{definition}

The idea of quenching and it's several applications can be found in  \cite{Levine83,Beauregard22}.

\begin{figure}[ht]
	{\scalebox{0.49}[0.49]{
			\includegraphics[width=\linewidth,height=4in]{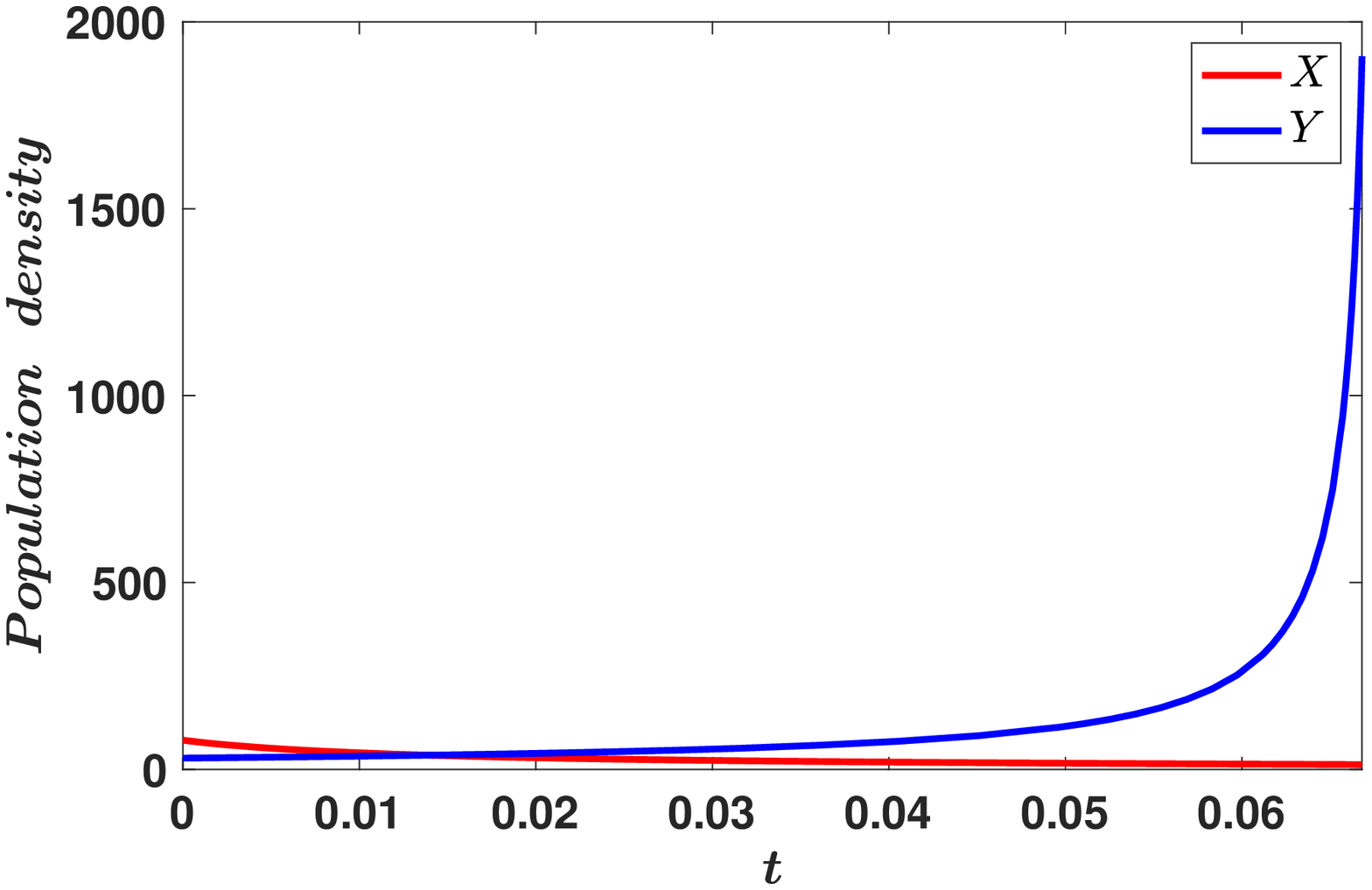}}}
	{\scalebox{0.49}[0.49]{
			\includegraphics[width=\linewidth,height=4in]{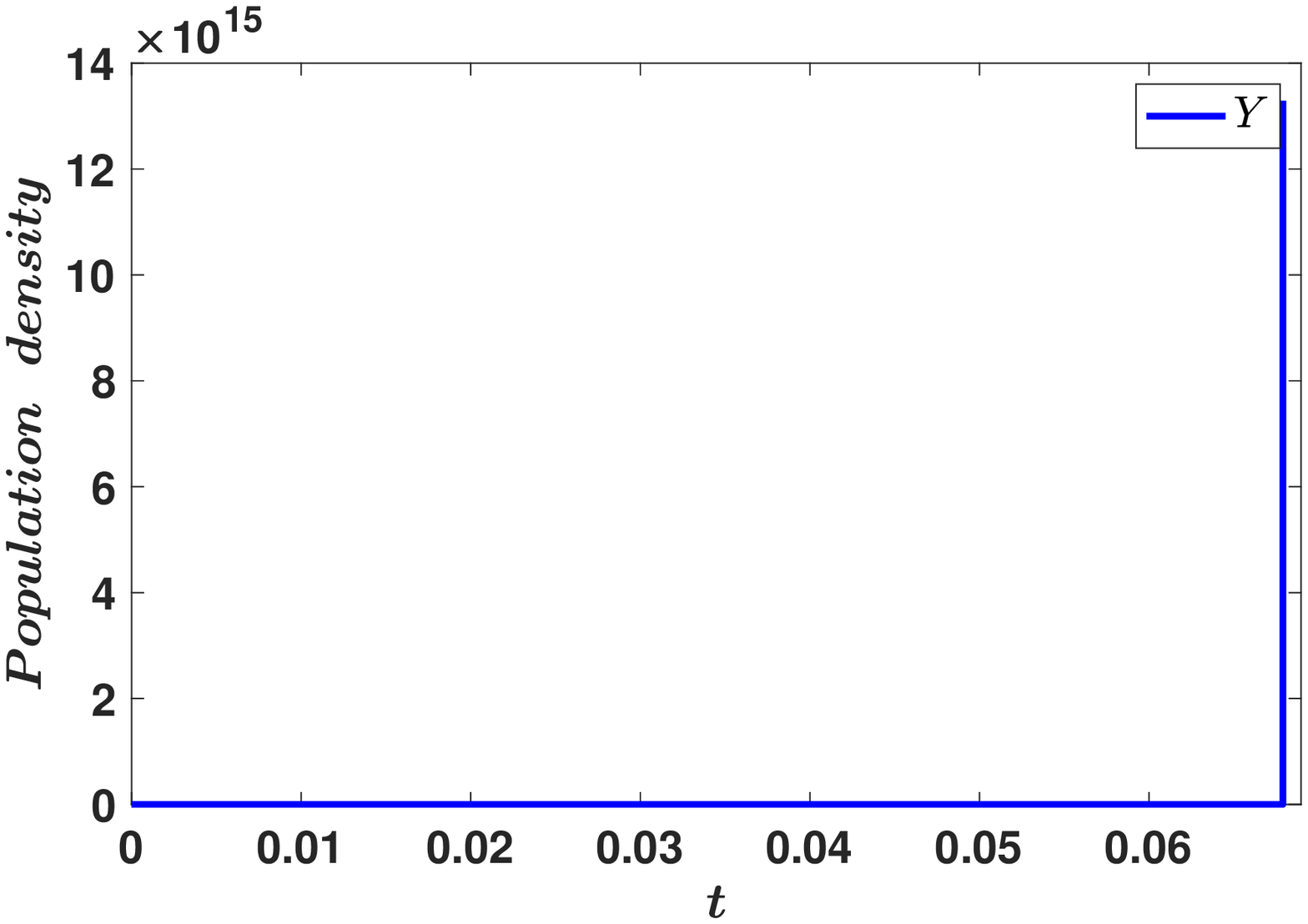}}}
		\caption{Time series simulation of the non-delayed predator-prey dynamics given by the mathematical model in \eqref{model}. The simulation used specific parameter values: $R=1$, $K=1$, $M=1.2$, $p=2$, $D=0.5$, $E=0.2$, $A=0.2$, and $C=0.3$, with initial data $[X(0),Y(0)]=[78,30]$. The parameter choices ensure that the model satisfies the requirements for a blow-up, as stated by Theorem \ref{ode_blowup}.}
	\label{fig:blowup}
\end{figure}

\begin{figure}[ht]
    {\scalebox{0.49}[0.49]{
			\includegraphics[width=\linewidth,height=4in]{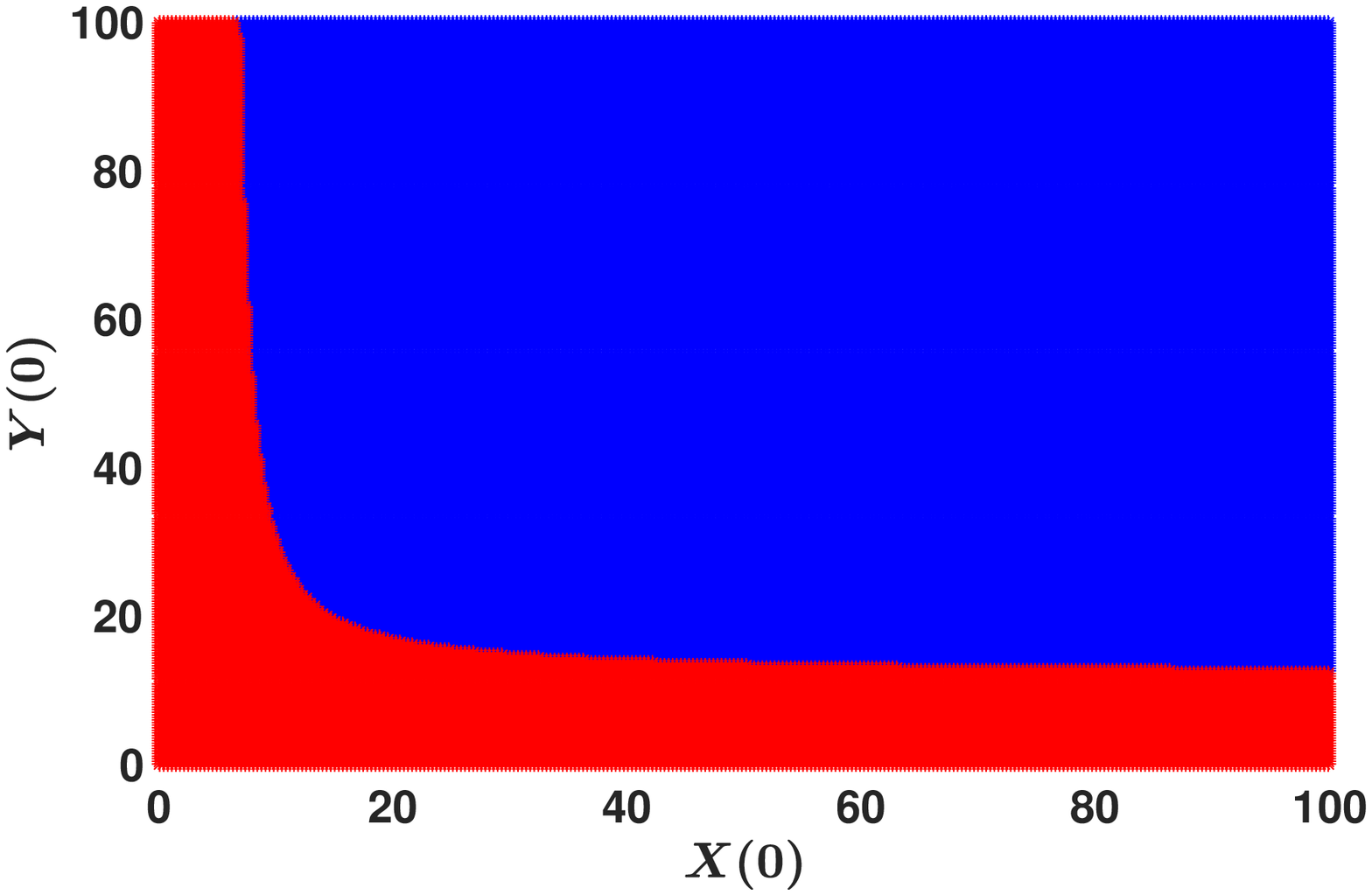}}}
	{\scalebox{0.49}[0.49]{
		\includegraphics[width=\linewidth,height=4in]{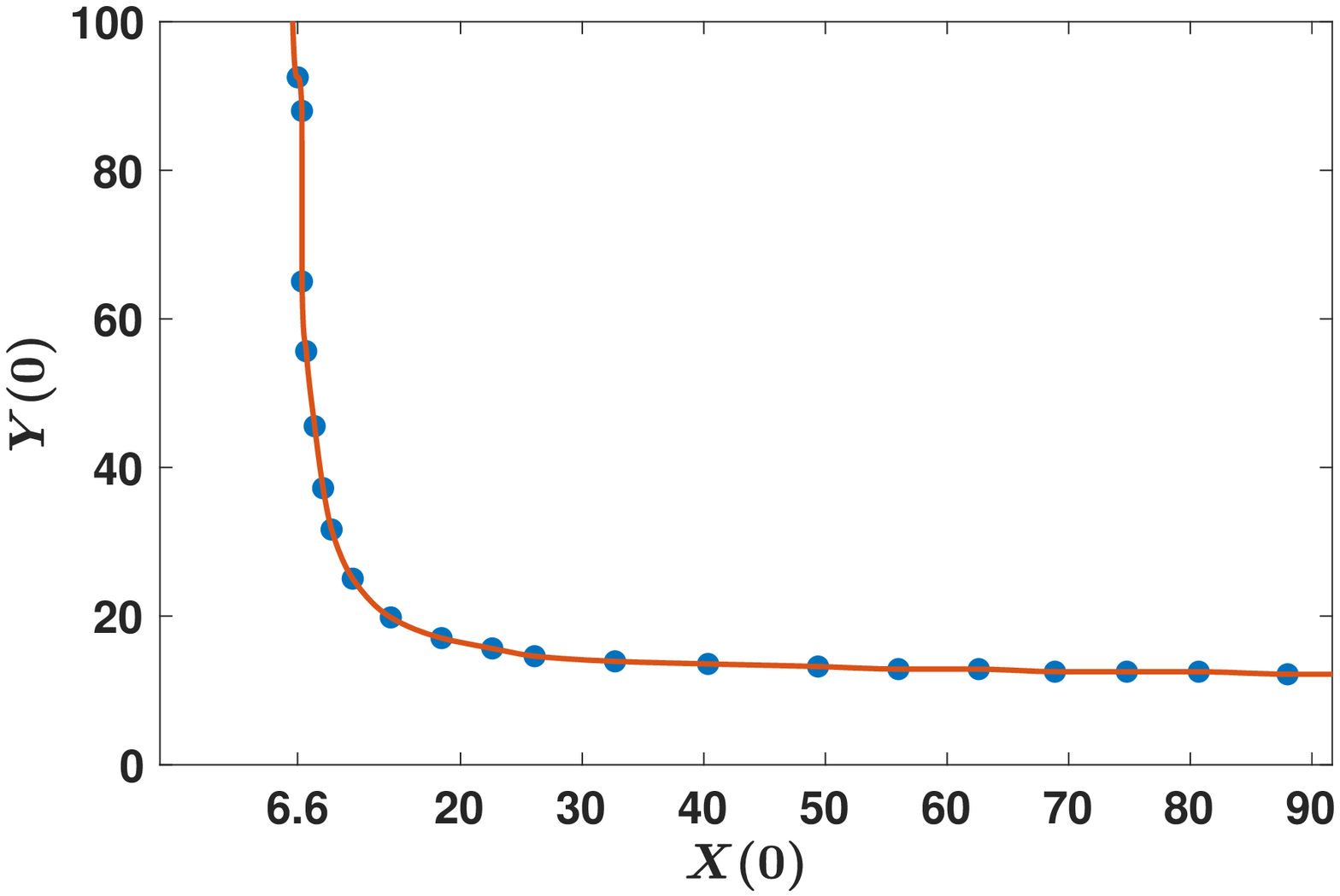}}}
    \caption{ (A) Dynamics of non-delayed predator-prey dynamics given by the mathematical model in \eqref{model} when certain initial conditions are applied. The colored regions indicate two possible outcomes: the red region represents a non-blowup case, while the blue region represents a blow-up case if we choose initial conditions from that region. The model was simulated using specific parameter values: $R=1$, $K=1$, $M=1.2$, $p=2$, $D=0.5$, $E=0.2$, $A=0.2$, and $C=0.3$. (B) Shape-preserving interpolation of the boundary of the region of attraction. The blow-up tolerance parameter is assumed as $10^8.$}
	\label{fig:blowup_ic}
\end{figure}


\begin{theorem}
\label{thm:thm2}
Consider the model \eqref{model}. Then $Y(t)$ blows up in finite time, that is

\begin{equation}
\label{eq:bu}
\mathop{\lim}_{t \rightarrow T^{*} < \infty} |Y(t)| \rightarrow \infty,
 \end{equation}

as long as the initial data $(X_{0},Y_{0})$ is sufficiently large.


\end{theorem}

\begin{proof}

Consider \eqref{model}, with positive initial conditions $%
(X_{0},Y_{0})$. By integrating the equation for the predator in time, we obtain%
\begin{equation*}
Y=\frac{1}{\dfrac{1}{Y_{0}}-\displaystyle\int_{0}^{t}\left( D-\frac{E}{%
X(s)+A}\right) ds}.
\end{equation*}%
Now consider the continuous function:
\begin{equation}
\psi \left( t\right) =\frac{1}{Y_{0}}-\int_{0}^{t}\left( D -\frac{%
E}{X(s)+A}\right) ds.
\end{equation}%
Our goal is to show $\psi(T) = 0$, at some time $T>0$. Now we know from Proposition \ref{prop.exis.blow.2.2} since solutions are classical and so continuous $\forall \hspace{.051in} T$, where $t < T_{max}$ and $T < T_{max}$, where $T_{max}$ is the maximal interval of existence of the local solutions,
that for a given $\delta >0$, $\delta < T_{max}$,
we can choose $X_{0}$ sufficiently large such that%
\begin{equation}
X+A> \dfrac{E}{D},\ \forall \hspace{.05in} t\in \lbrack 0,\delta ),
\end{equation}%

and so
\begin{equation}
\frac{E}{X+A}<D\ \forall \hspace{.05in} t\in \lbrack
0,\delta ).
\end{equation}%
This implies
\begin{equation*}
\int_{0}^{t}\left( \frac{E}{X+A}\right) ds<\int_{0}^{t}Dds=
D t, \  \text{\ for all }t \in \lbrack 0,\delta ).
\end{equation*}%

This entails, that for some $T^{*} < \delta$,
\begin{eqnarray}
\psi \left( T^{*}\right) &=& \frac{1}{Y_{0}}-\int_{0}^{T}\left(D-\frac{E}{X(s)+A}\right) ds  \nonumber \\
&=& \frac{1}{Y_{0}} + \int_{0}^{T}\left(\frac{E}{X(s)+A}\right) ds  - DT \nonumber \\
&<&  \frac{1}{Y_{0}} + \int_{0}^{\delta}\left(\frac{E}{X_{0}}\right) ds  - D\delta   \nonumber \\
&=& \frac{1}{Y_{0}} + \left(\frac{\delta E}{X_{0}}\right) - D\delta \leq 0. \nonumber \\
\end{eqnarray}

The last inequality follows by choosing $(X_{0},Y_{0})$ appropriately large. Thus $\psi \left( T^{*}\right) \leq 0$, since $ \psi \left(0 \right) = Y_{0 }> 0$, a simple application of the intermediate value theorem tells us that there exists a $T^{**} < T^{*}$, where $\psi \left( T^{**}\right)= 0$, and so $Y$ has blown-up at the finite time $T^{**}$, by a standard comparison argument \cite{Hen}.
\end{proof}

A corollary to this is,

\begin{corollary}
\label{cor:bdd_patra11}
The positive quadrant $\mathbb{R}^{2}$ is not invariant for system \eqref{model}.
\end{corollary}

\begin{Remark}
We next provide a counter-example to the boundedness claim via Theorem 3.2 \cite{patra2022effect}.  Pick $R=1, K=1, M=1.2 ,p=2, D=0.5, E=0.2, A=0.2$ and $C=0.3$, then $X^*=0.2$ and $ Y^*=0.226667$. Let's calculate the fractions involved in parametric restrictions given by \eqref{res},
\[ 0.2 = \min \{1.2,0.2\} < \dfrac{0.5 \times 0.2}{(0.2)^2 + 0.3}=0.294118,\]
and
\[ 0=0.5- \dfrac{0.2}{0.2 + 0.2} < \dfrac{0.5}{1.2}\Big( \dfrac{0.2}{0.226667}\Big)^2=0.324585\]

This choice of parameter satisfies the constraints of Theorem \ref{thm:bdd_patra}, but from the time-series plot Figure~\ref{fig:blowup}, the predator $Y$ blows up to infinity in finite time $t=T^*=6.789603 \times 10^{-2}.$
\end{Remark}

\begin{Remark}
    For the parameters: $R=1, K=1, M=1.2 ,p=2, D=0.5, E=0.2, A=0.2$ and $C=0.3$, the parameters satisfies the parametric restriction given in Theorem \ref{thm:stable_int}:
    \[ 0.3=C>\frac{\left(\frac{E}{D}-A\right)^p (A D p+A D+D K p-E p-E)}{E-A D}=0.28.\]
   The interior equilibrium $(X^*,Y^*)=(0.2,0.226667)$ is locally asymptotic stable.

\end{Remark}

\begin{Remark}

Note, via a simple comparison to the problem $\frac{dY}{dt}=DY^{2}$, the blow-up time $T^{**}$ for \eqref{model}, has a lower estimate given by, $\frac{1}{D Y_{0}} < T^{*}$.
Also via standard theory,
$T_{max} < T^{**}$. Here $T_{max}$ was the maximal interval of existence assumed on $%
Y(t) $, so that we could make the formal estimates.

\end{Remark}

\begin{figure}[ht]
	{\scalebox{0.49}[0.49]{
			\includegraphics[width=\linewidth,height=4in]{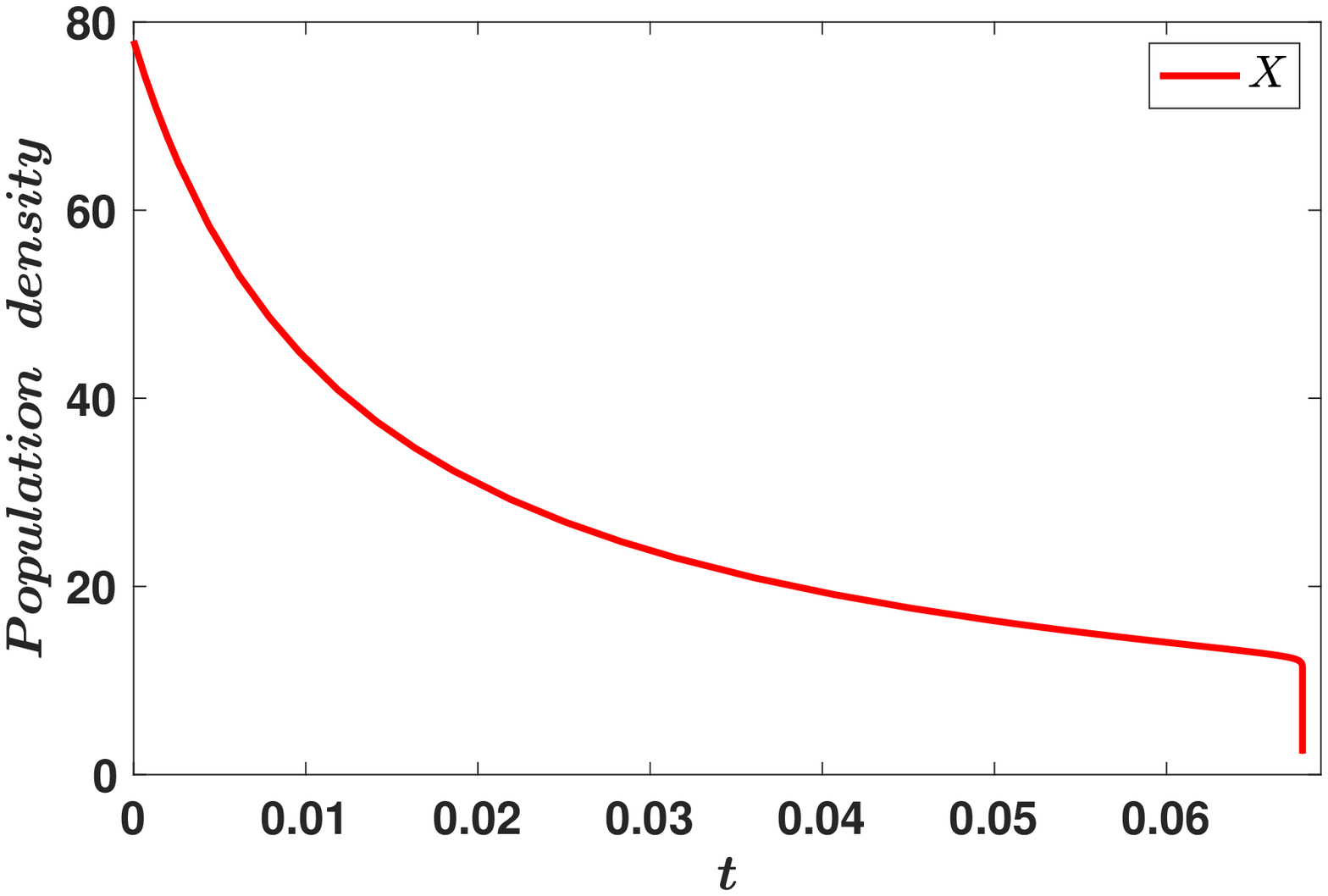}}}
	{\scalebox{0.49}[0.49]{
			\includegraphics[width=\linewidth,height=4in]{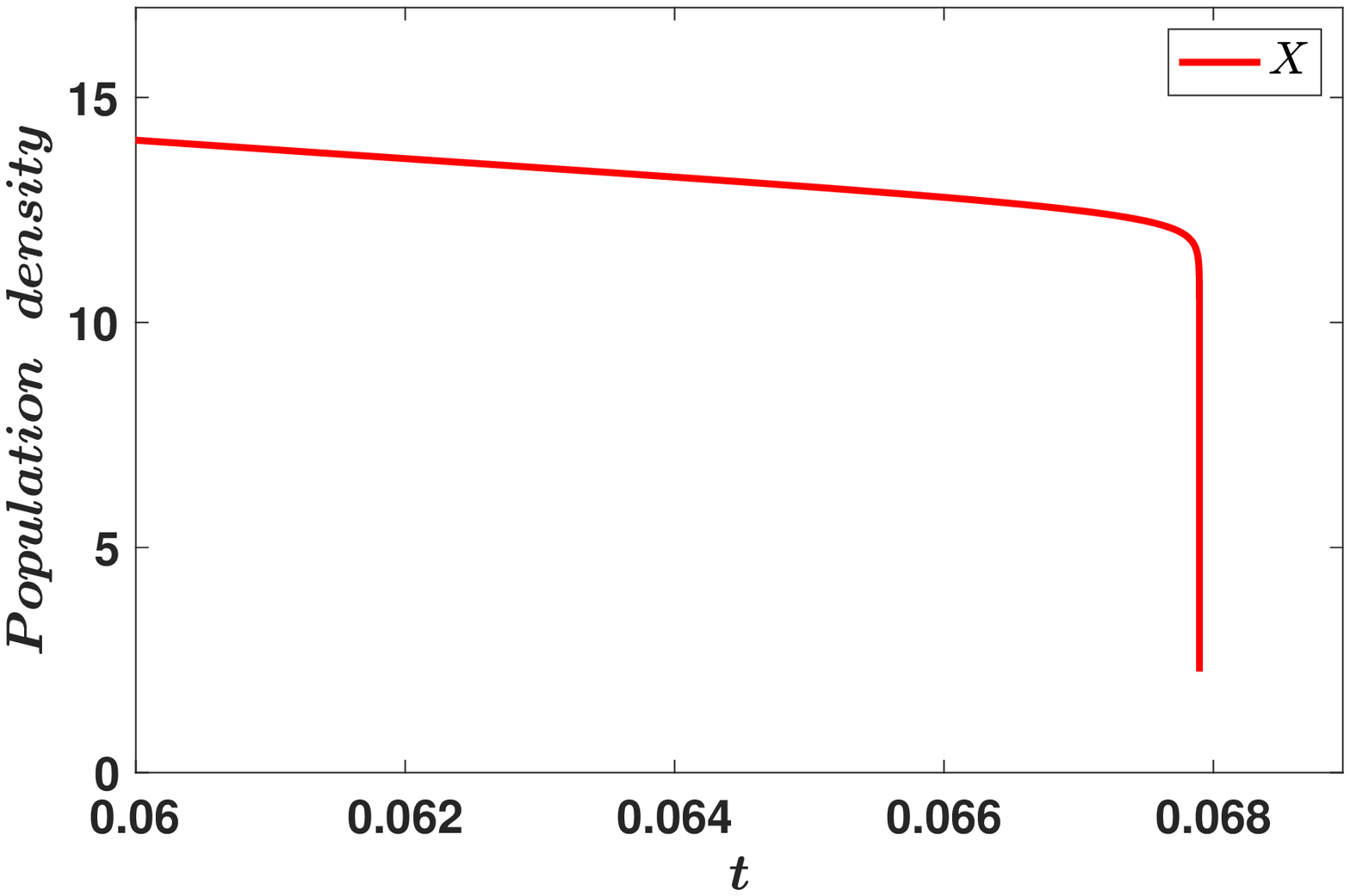}}}
		\caption{Time series simulation of the predator-prey dynamics given by the mathematical model in \eqref{model}. The simulation used specific parameter values: $R=1$, $K=1$, $M=1.2$, $p=2$, $D=0.5$, $E=0.2$, $A=0.2$, and $C=0.3$, with initial data $[X(0),Y(0)]=[78,30]$. }
	\label{fig:quench}
\end{figure}

\begin{Remark}\label{rem:quen1}
Let's numerically validate that predator-prey dynamics given by the mathematical model in \eqref{model} quenches in finite time. Consider the parameters as $R=1, K=1, M=1.2 ,p=2, D=0.5, E=0.2, A=0.2$ and $C=0.3$. From the simulation, the blow up time $T^*=6.789603 \times 10^{-2}$, and $X(T^*)=2.2434$ and $Y(T^*)=1.328 \times 10^{16}$. Hence,
\[ \displaystyle\frac{d X}{d t} = -6.7078 \times 10^{15} \quad \& \quad \displaystyle\frac{d Y}{d t} =7.3834 \times 10^{31}.\]
\end{Remark}

This motivates the following Theorem:

\begin{theorem}
\label{conj:quench1}
    Consider the system \eqref{model}, for any choice of parameters and initial conditions chosen sufficiently large, the prey population $X(t)$ will quench at a finite time $T^{*}$, which is also a blow-up time for predator population $Y(t)$, i.e.,
    \begin{equation*}
\underset{t \rightarrow T^{\ast} < \infty}{\lim }\left\vert \frac{dX\left( t\right)}{dt}
\right\vert =+\infty \quad \& \quad \mathop{\lim}_{t \rightarrow T^{*} < \infty} |Y(t)| \rightarrow \infty,
\end{equation*}%
\end{theorem}

This is proved via a series of two lemmas.

\begin{lemma}
\label{lem:quench11}
    Consider the system \eqref{model}, for any choice of parameters and initial conditions chosen sufficiently large, the prey population $X(t)$ will quench at a finite time $T^{**}$, which is greater than or equal to the blow-up time $T^*$ for predator population $Y(t)$.
\end{lemma}

\begin{proof}
    From the Theorem \ref{thm:thm2}, we have the $Y(t)$ blows up in finite time for sufficiently large initial data $(X_{0},Y_{0})$, that is
\[ \mathop{\lim}_{t \rightarrow T^{*} < \infty} |Y(t)| \rightarrow \infty.\]

We now claim that $T^{\ast \ast}$ the finite-time quenching for prey population is greater than or equal to the blow-up time $T^*$ for predator population $Y(t)$, i.e., $T^{\ast \ast} \ge T^\ast$. To claim this it suffices that 
$\underset{t \rightarrow T^{\ast} < \infty}{\lim}X(t) \neq 0.$ Let's on contrary assume that 
\[\mathop{\lim}_{t \rightarrow T^{*} < \infty} X(t)= 0 \implies \text{From Proposition } \ref{prop.exis.blow.2.2}, \text{ we have}  \quad X(t) \equiv 0 ,\quad\forall \hspace{.05in} \Omega_\epsilon=(T^\ast -\epsilon,T^\ast). \]

Consider, the $Y$ nullclines in $\Omega_\epsilon$, we have
\[ \frac{d Y}{d t}=\left( D- \dfrac{E}{A} \right) Y^2 <0.\]
Therefore, $Y$ is decreasing in $\Omega_\epsilon$, which contradicts the fact that $Y(t)$ blows up in finite time.

Hence, 
\[ \mathop{\lim}_{t \rightarrow T^{*} < \infty} X(t)= \widehat{x} (\neq 0) \quad \& \quad \mathop{\lim}_{t \rightarrow T^{*} < \infty} Y(t) = +\infty.\]
Therefore, we get 
\[ \underset{t \rightarrow T^{\ast} < \infty}{\lim }\left\vert \frac{dX\left( t\right)}{dt}
\right\vert = \underset{t \rightarrow T^{\ast} < \infty}{\lim } \Big\{ R X\left(1-\displaystyle\frac{X(t)}{K}\right) - \dfrac{M X Y}{X^p + C} \Big\}= +\infty. \]
Hence, we have that $T^{\ast \ast}$ the finite-time quenching for prey population is greater than or equal to the blow-up time $T^*$ for predator population $Y(t)$.
\end{proof}
Next we show,

\begin{lemma}
\label{lem:quench12}
    Consider the system \eqref{model}, for any choice of parameters and initial conditions chosen sufficiently large, the prey population $X(t)$ will quench at a finite time $T^{**}$, which is less than or equal to the blow-up time $T^*$ for predator population $Y(t)$.
\end{lemma}

\begin{proof}
We prove this via contradiction. Assume to the contrary that the quenching time of the prey, $T^{**} > T^{*}$, where $T^{*}$ is the blow-up time of the predator.
   Since $T^{\ast \ast}$ is the finite quenching time, we have $\underset{t \rightarrow T^{\ast \ast} < \infty}{\lim }\left\vert \frac{dX\left( t\right)}{dt}
\right\vert =+\infty$. From Proposition \ref{prop.exis.blow.2.2}, and using the fact that the prey population is bounded from above, we have
\[\underset{t \rightarrow T^{\ast \ast} < \infty}{\lim }  R X\left(1-\displaystyle\frac{X(t)}{K}\right) < \infty.\]
Upon using this information in the prey equation equation, and passing the limit, we get
\[\underset{t \rightarrow T^{\ast \ast} < \infty}{\lim }  \dfrac{X Y}{X^p + C} = +\infty \quad \implies \underset{t \rightarrow T^{\ast \ast} < \infty}{\lim } \quad Y(t) = +\infty .\]
This gives us a contradiction, because if $T^{**} > T^{*}$, where we assume that $T^\ast$ is the finite time of blowup, then the solution will cease to exist past $T^\ast$.

This shows that the prey population $X(t)$ will quench at a finite time $T^{**}$, which is less than or equal to the blow-up time $T^*$ for predator population $Y(t)$.
\end{proof}

The Lemmas \ref{lem:quench11} and \ref{lem:quench12}, used in conjunction prove Theorem \ref{conj:quench1}, giving the equivalence of $T^{**}$ and $T^{*}$.

\section{Delayed ODE}

Adding the delay parameter $\tau$, to the system \eqref{model}, the delayed model is given by 
\begin{equation}
	\left\{\begin{array}{l}
		\displaystyle\frac{d X}{d t}=X\left(1-\displaystyle X(t-\tau)\right) - \dfrac{M X Y}{X^p + C}\vspace{2ex}\\
		\displaystyle\frac{d Y}{d t}=\left( D- \dfrac{E}{X + A} \right) Y^2
	\end{array}\right.
	\label{delay_model}
\end{equation}
subject to initial conditions
\begin{align*}
    X(\theta)&=\phi_1 (\theta) >0, \\
    Y(\theta)&=\phi_2 (\theta) >0, \\
    \theta & \in [-\tau,0); \quad \phi_i (0)>0, \quad i=1,2.
\end{align*}

\begin{theorem}
\label{thm2dm}
Consider the model \eqref{delay_model}. Then


 $Y(t)$ blows up in finite time, that is

\begin{equation}
\label{eq:bu1}
\mathop{\lim}_{t \rightarrow T^{*} < \infty} |Y(t)| \rightarrow \infty,
 \end{equation}

as long as the initial data $(X_{0},Y_{0})$ is sufficiently large.
\end{theorem}

\begin{proof}
The proof follows similarly as in the proof of Theorem \ref{thm:thm2}.
    
\end{proof}

\begin{figure}[ht]
	{\scalebox{0.49}[0.49]{
			\includegraphics[width=\linewidth,height=4in]{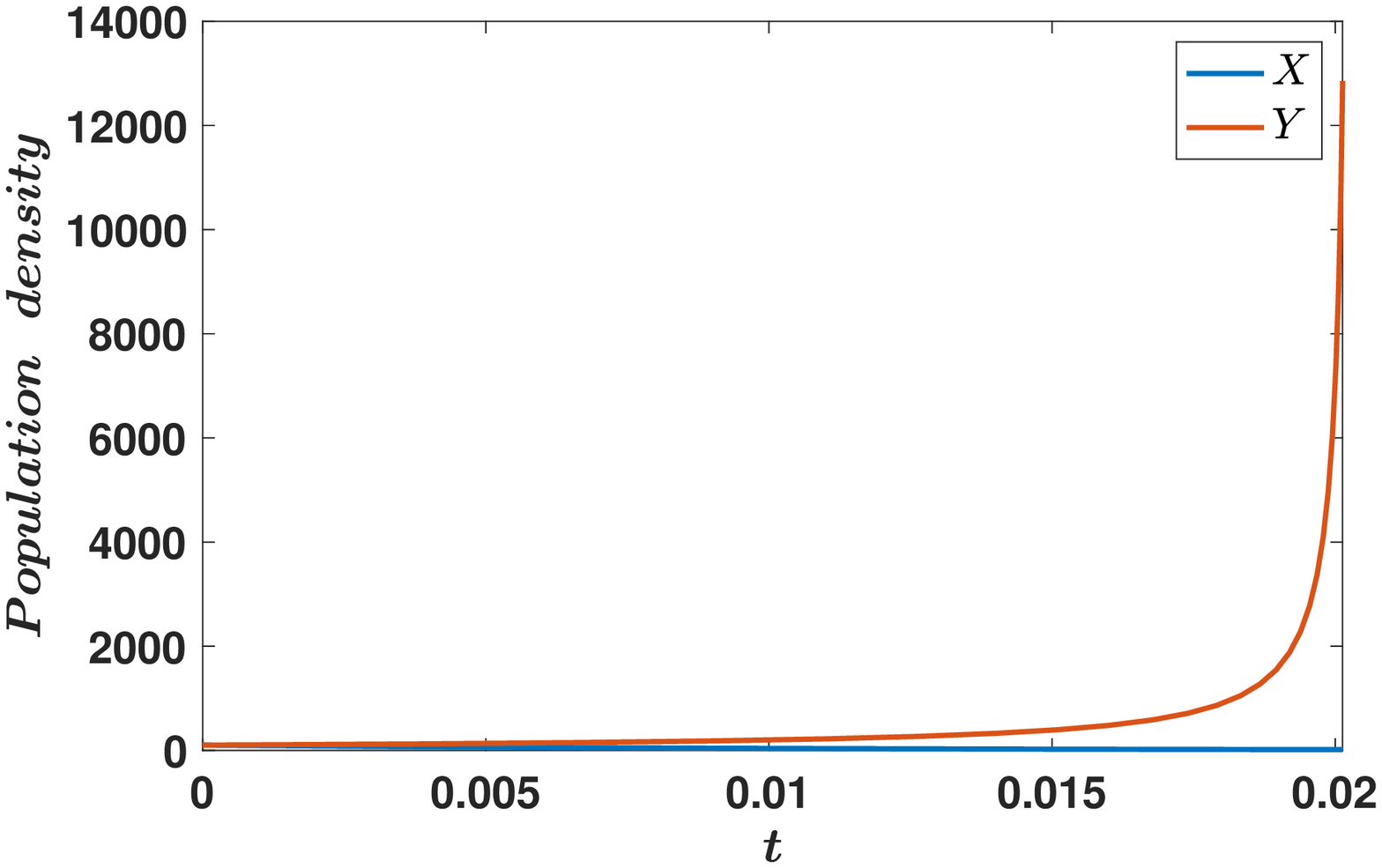}}}
	{\scalebox{0.49}[0.49]{
			\includegraphics[width=\linewidth,height=4in]{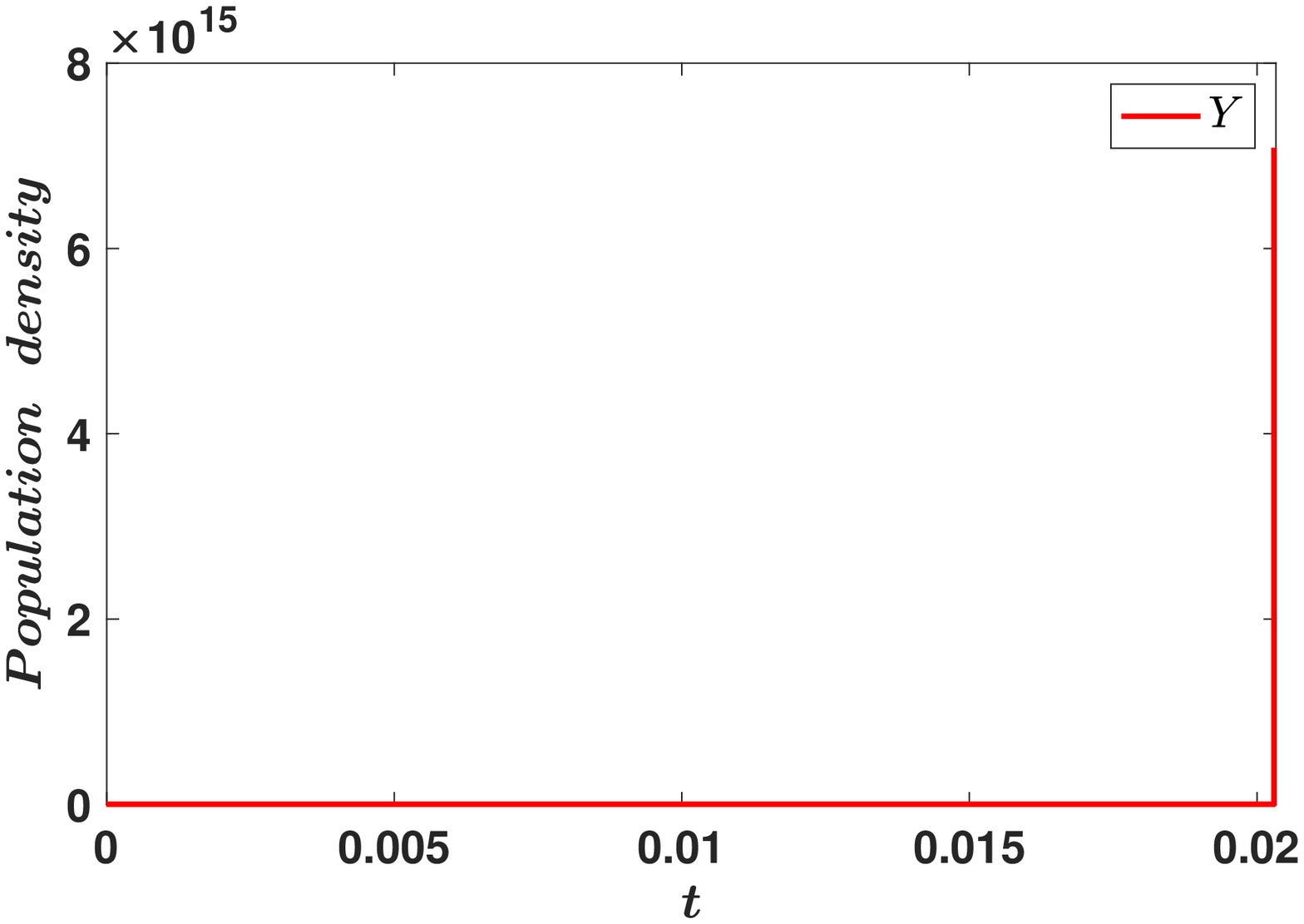}}}
		\caption{Time series simulation of the delayed predator-prey dynamics given by the mathematical model in \eqref{delay_model}. The simulation used specific parameter values: $R=1$, $K=1$, $M=1.2$, $p=2$, $D=0.5$, $E=0.2$, $A=0.2, \tau=1$, and $C=0.3$, with initial data $[X(0),Y(0)]=[100,100]$. The parameter choices ensure that the model satisfies the requirements for a blow-up, as stated by Theorem \ref{delay_ode_blowup}.}
	\label{fig:blowup_delay}
\end{figure}

We now derive a delayed version of \eqref{model} s.t. there is global in-time existence of solutions for all initial conditions. We add the delay parameter $\tau$, to the system \eqref{model}, by considering a gestation effect in the predator, 
\begin{equation}
	\left\{\begin{array}{l}
		\displaystyle\frac{d X}{d t}=X\left(1-\displaystyle X\right) - \dfrac{M X Y}{X^p + C}\vspace{2ex}\\
		\displaystyle\frac{d Y}{d t}= DY(t-\tau)^2 - \dfrac{E}{X + A}  Y^2
	\end{array}\right.
	\label{delay_model1}
\end{equation}
subject to initial conditions
\begin{align*}
    X(\theta)&=\phi_1 (\theta) >0, \\
    Y(\theta)&=\phi_2 (\theta) >0, \\
    \theta & \in [-\tau,0); \quad \phi_i (0)>0, \quad i=1,2.
\end{align*}

\begin{theorem}\label{delay_ode_blowup}
    Consider the system (\ref{delay_model1}), for any choice of a parameter such that $\delta_1>0$, and any initial condition  the predator population will remain bounded for all time.
\end{theorem}

\begin{proof}
We proceed by contradiction. Assume the predator population blows up at a finite time $t=T^{*}$. Let us integrate the predator equation in the time interval $[T^{*}-1,T^{*}]$. Thus we obtain

\begin{eqnarray}
&& Y(T^{*}) = D\int^{T^{*}}_{T^{*}-1}Y(t-\tau)^2dt - \int^{T^{*}}_{T^{*}-1}\dfrac{E}{X + A}  Y^2dt +  Y(T^{*}-1)\nonumber \\
&& < D\int^{T^{*}}_{T^{*}-1}Y(t-\tau)^2 +Y(T^{*}-1)\nonumber \\
&& = D\int^{T^{*}-\tau}_{T^{*}- 1 - \tau}Y(t)^2 dt +Y(T^{*}-1) \nonumber \\
\end{eqnarray}
this follows via the positivity of solutions and a simple change of time variable. Now using standard estimates, we obtain,

\begin{equation}
Y(T^{*}) \leq D \left(\sup_{T^{*}- 1 - \tau \leq t \leq T^{*}-\tau} Y(t)\right) [(T^{*}-\tau)-(T^{*}-\tau-1)]
    \end{equation}

Via the assumption that blow-up occurs at $t=T^{*}$, we have that 
\begin{equation}
Y(T^{*}) = \infty \leq D \left( \sup_{T^{*}- 1 - \tau \leq t \leq T^{*}-\tau}Y(t) \right) < \infty
    \end{equation}
Thus we have a contradiction. This proves the theorem.
    
\end{proof}

\begin{figure}[ht]
	{\scalebox{0.49}[0.49]{
			\includegraphics[width=\linewidth,height=4in]{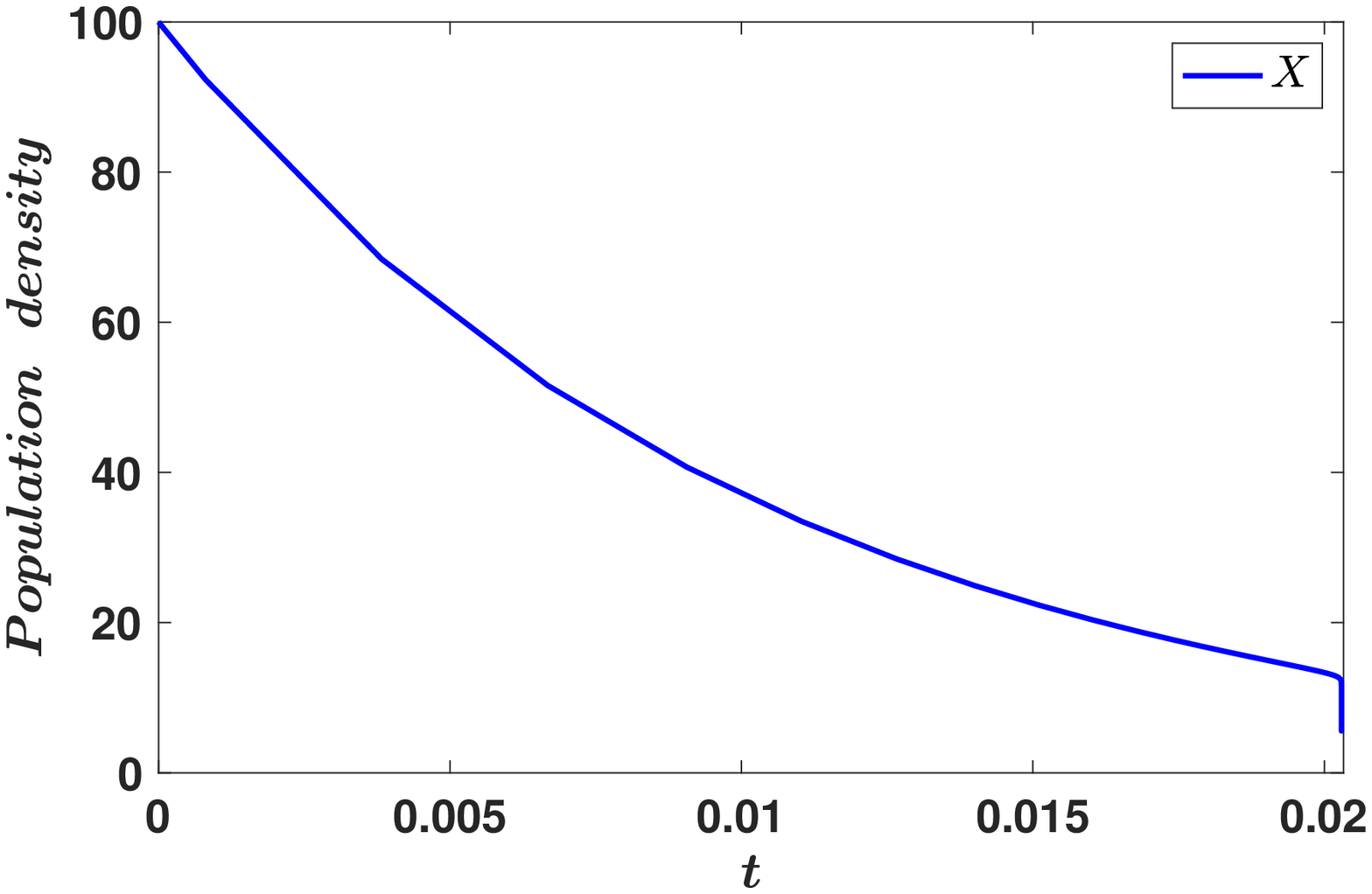}}}
	{\scalebox{0.49}[0.49]{
			\includegraphics[width=\linewidth,height=4in]{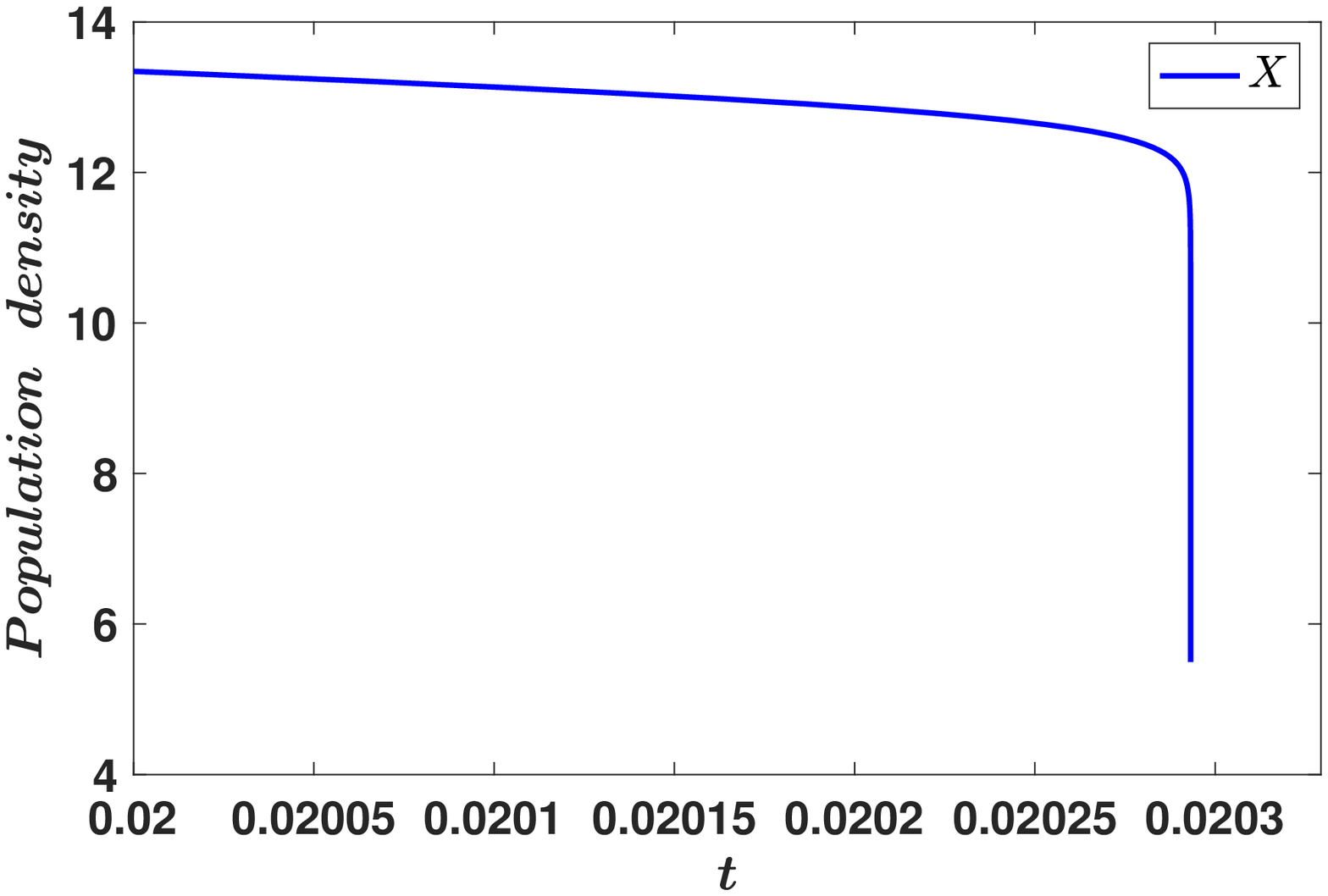}}}
		\caption{Time series simulation of the predator-prey dynamics given by the mathematical model in \eqref{delay_model}. The simulation used specific parameter values: $R=1$, $K=1$, $M=1.2$, $p=2$, $D=0.5$, $E=0.2$, $A=0.2, \tau=1$, and $C=0.3$, with initial data $[X(0),Y(0)]=[100,100]$. }
	\label{fig:quench_delay}
\end{figure}

\begin{Remark}\label{rem:quen2}
Let's numerically validate that predator-prey dynamics given by the mathematical model in \eqref{model} quenches in finite time. Consider the parameters as: $R=1$, $K=1$, $M=1.2$, $p=2$, $D=0.5$, $E=0.2$, $A=0.2, \tau=1$, and $C=0.3$. From the simulation, the blow up time $T^*=2.029316 \times 10^{-2}$, and $X(T^*)=5.4949$ and $Y(T^*)=7.086 \times 10^{15}$. Hence,
\[ \displaystyle\frac{d X}{d t} = -1.532 \times 10^{15} \quad \& \quad \displaystyle\frac{d Y}{d t} =2.3347\times 10^{31}.\]
\end{Remark}

This motivated us to present a Theorem:

\begin{theorem}
\label{conj:quench2}
    Consider the system \eqref{delay_model}, for any choice of a parameter such that the initial condition is sufficiently large, then the prey population $X(t)$ will quench in finite time $T^*$, which is also a blow-up time for predator population $Y(t)$, i.e.,
    \begin{equation*}
\underset{t \rightarrow T^{\ast} < \infty}{\lim }\left\vert \frac{dX\left( t\right)}{dt}
\right\vert =+\infty \quad \& \quad \mathop{\lim}_{t \rightarrow T^{*} < \infty} |Y(t)| \rightarrow \infty,
\end{equation*}%
\end{theorem}
\begin{proof}
The proof follows similarly as in the proof of Theorem \ref{conj:quench1}.
    
\end{proof}
\section{Feedback Control}
This section shows that the linear feedback model proposed in  \cite{patra2022effect} to control the system instability is not enough to prevent blow-up for both the delayed and non-delayed models. Consider the linear feedback model proposed in  \cite{patra2022effect}.

\begin{equation}\label{f1model}
	\left\{\begin{array}{l}
		\displaystyle\frac{d X}{d t}= X\left(1- X (t-\tau ) \right) - \dfrac{M X Y}{X^p + C} - u (X-X^*)\vspace{2ex}\\
		\displaystyle\frac{d Y}{d t}=\left( D- \dfrac{E}{X + A} \right) Y^2 - u (Y-Y^*)
	\end{array}\right.
\end{equation}

where $(X^*,Y^*)$ are the interior equilibrium point of system \eqref{f1model} when linear feedback control parameter $u$ is set to be zero. 

The authors in  \cite{patra2022effect} justified that the interior equilibrium for non-delayed \eqref{f1model} is stable if the parameters satisfy
\begin{align}\label{45}
    u > \dfrac{1}{2} \Big( A_{11} - X^* \Big) \quad \& \quad u^2 - (A_{11} -X^* ) u + A_{12} A_{21} >0,
\end{align}
where 
\begin{align}\label{Jaco_param}
    A_{11} = \dfrac{M p Y^* (X^*)^p}{((X^*)^p + C)^2}, \quad A_{12}=\dfrac{m X^*}{(X^*)^p +C} ,\quad \& \quad A_{21} = \dfrac{D^2 (Y^*)^2}{E}.  
\end{align}

\begin{Remark}\label{rem:feed1}
    Let's consider a few parameters that satisfy the parametric restriction given by \eqref{45}, i.e.,
    \[ M=1.2, \quad p=2, \quad C=0.3, \quad D=0.4, \quad E=0.2, \quad A=0.2 \quad \& \quad u=0.02.\]
    Figure~\ref{fig:blowup_non_delay_f1} shows that the solution blows up in a finite time.
\end{Remark}

Moreover, the authors in \cite{patra2022effect} also argued about asymptotic stability of interior for delayed \eqref{f1model} is stable if the parameters for $\lambda=i \omega $ simultaneously satisfies
\begin{equation}\label{48}
	\begin{array}{l}
		u^2 - A_{11} u + A_{12} A_{21} + X^* u \cos (\omega_0 \tau) - \omega_0^2 + X^* \omega_0 \sin (\omega_0 \tau) =0,\\
    (2u-A_{11}) \omega_0 + X^* \omega_0 \cos (\omega_0 \tau) - X^* u \sin (\omega_0 \tau)>0,
	\end{array}
\end{equation}

where $A_{11},A_{12},A_{21},$ and $A_{22}$ are given by \eqref{Jaco_param}.
\begin{Remark}\label{rem:feed2}
    Let's consider a few parameters that satisfy the parametric restriction given by \eqref{48}, i.e.,
    \[ M=1.2, \quad p=2, \quad C=0.3, \quad D=0.4, \quad E=0.2, \quad A=0.2, \quad \tau=2 \quad \& \quad u=0.06.\]
    From Figure~\ref{fig:blowup_delay_f1}, it is evident that the solution blows up in finite time.
\end{Remark}

\begin{figure}[ht]
	{\scalebox{0.49}[0.49]{
			\includegraphics[width=\linewidth,height=4in]{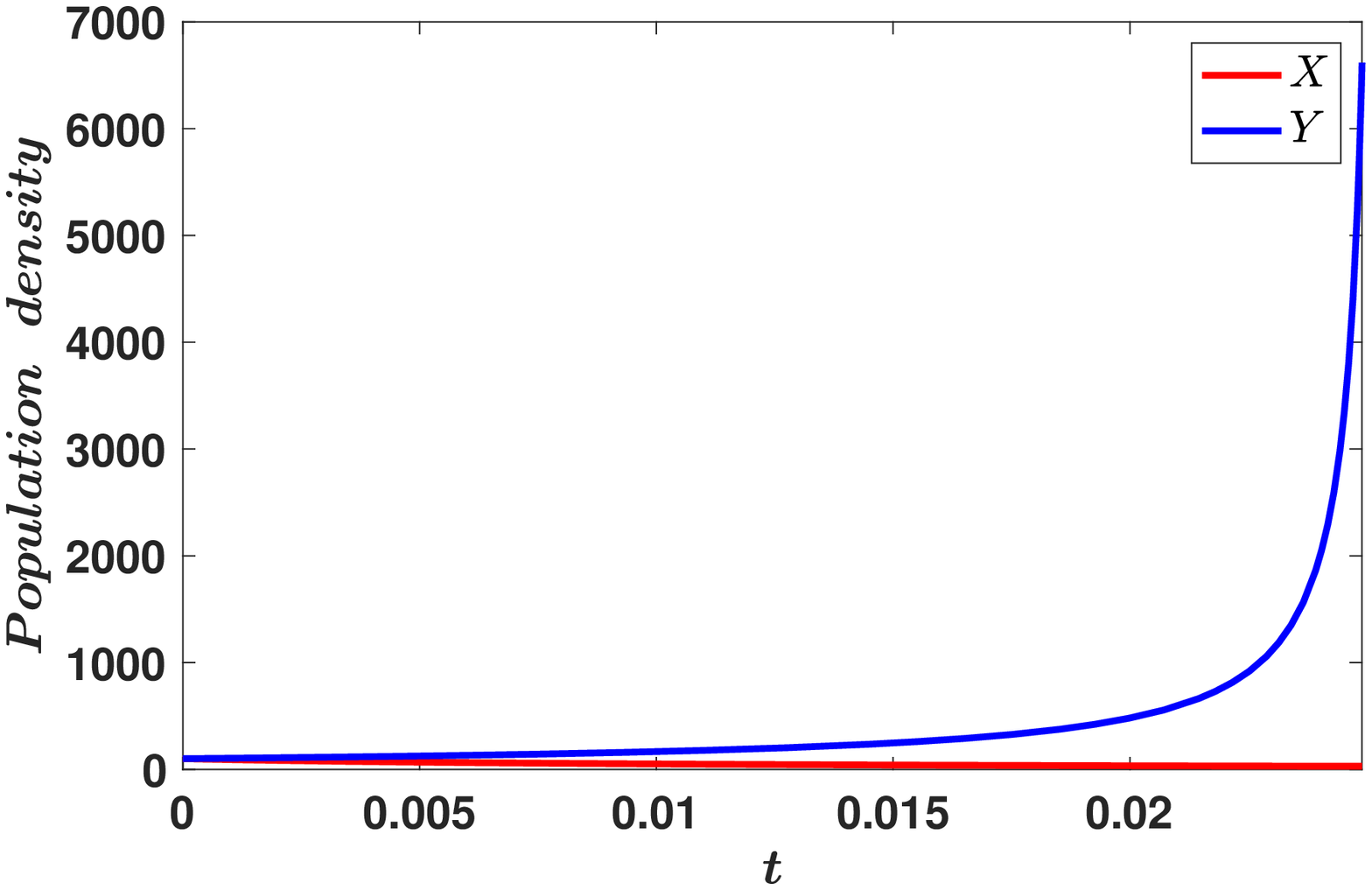}}}
	{\scalebox{0.49}[0.49]{
			\includegraphics[width=\linewidth,height=4in]{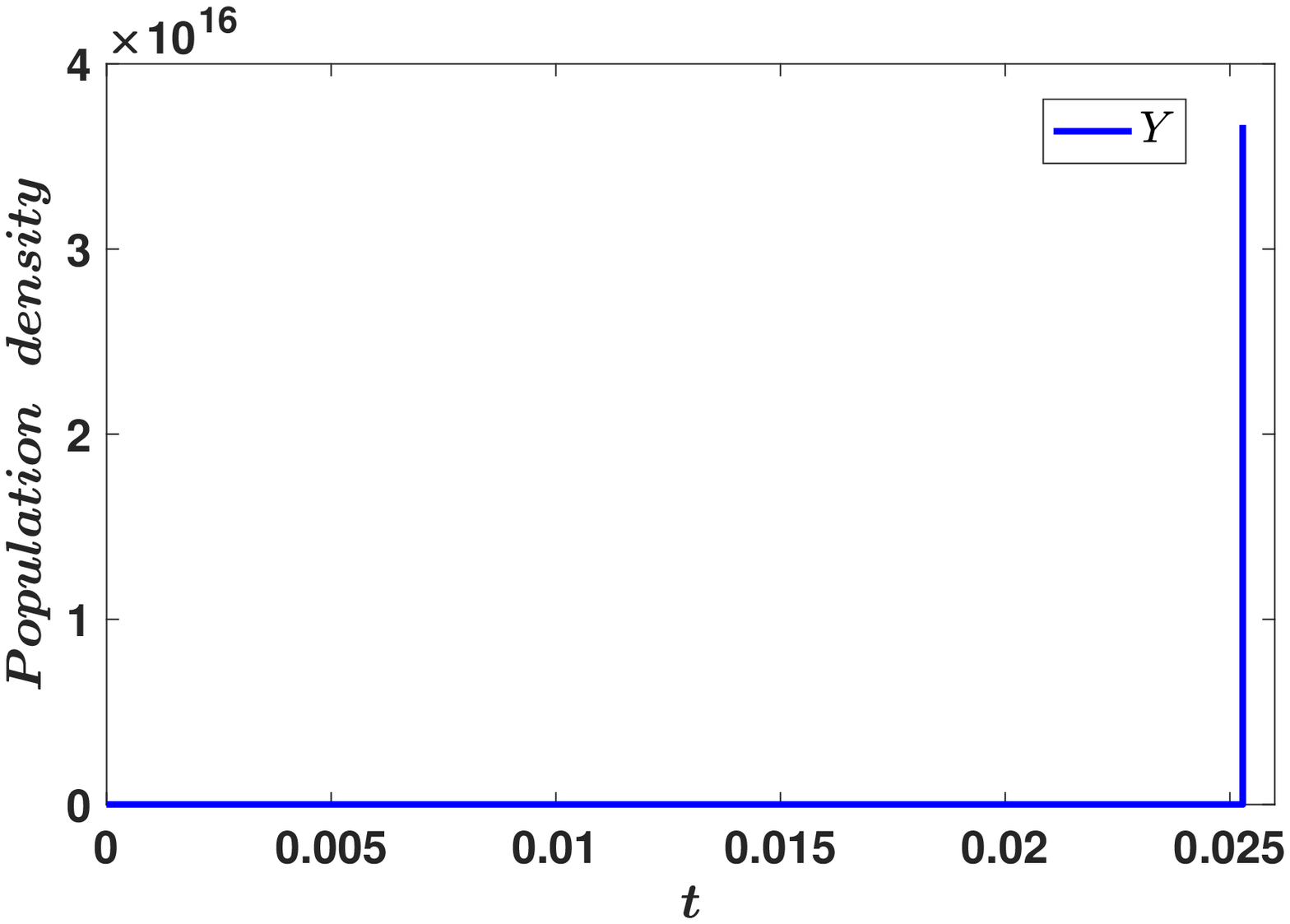}}}
		\caption{Time series simulation of the non-delayed predator-prey dynamics given by the mathematical model in \eqref{f1model}. The simulation used specific parameter values: $M=1.2$, $p=2$, $D=0.4$, $E=0.2$, $A=0.2, u =.02, \tau=0$ and $C=0.3$, with initial data $[X(0),Y(0)]=[100,100]$.}
	\label{fig:blowup_non_delay_f1}
\end{figure}

\begin{figure}[ht]
   { \scalebox{0.49}[0.49]{
			\includegraphics[width=\linewidth,height=4in]{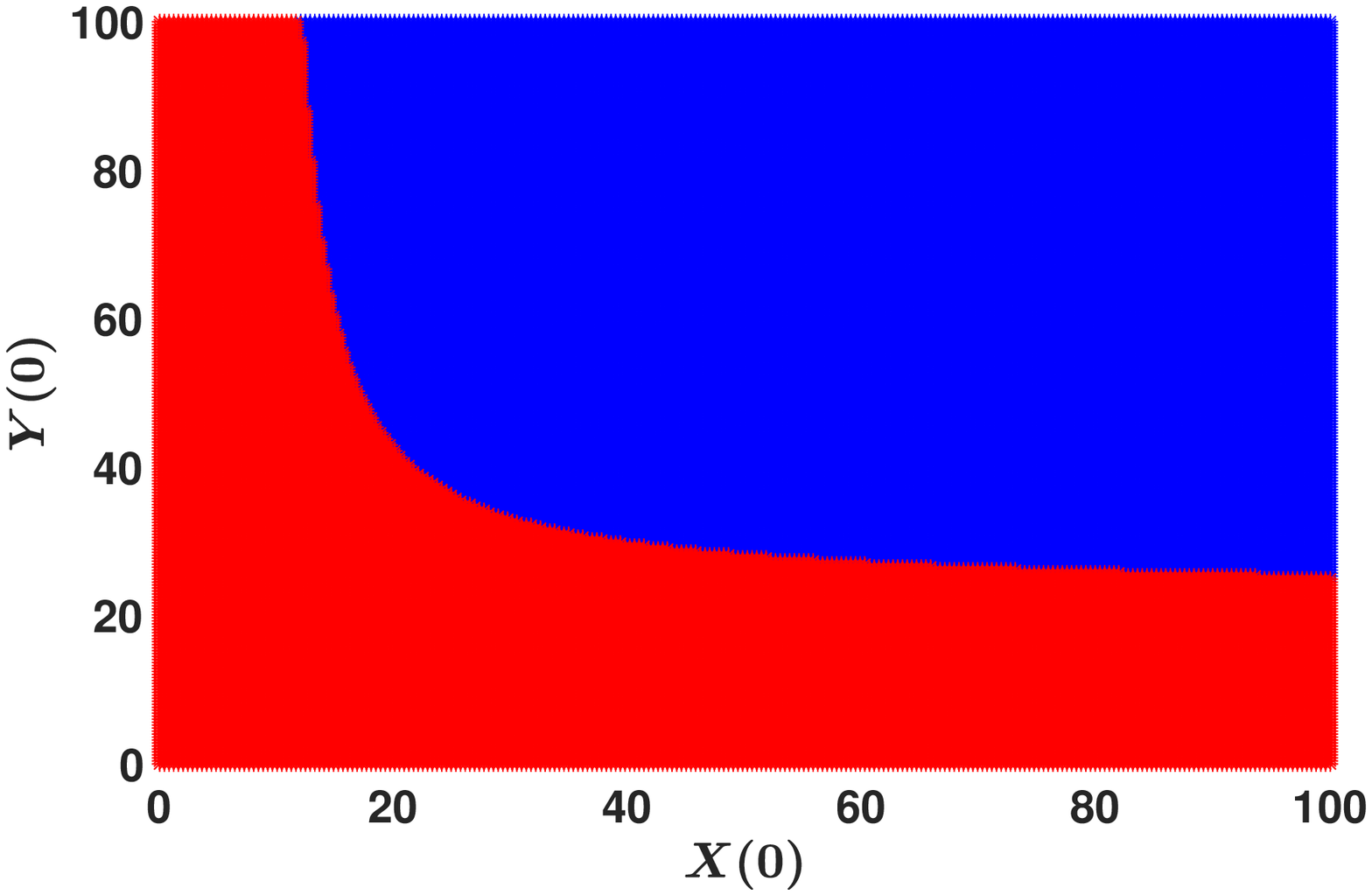}}}
   {\scalebox{0.49}[0.49]{
			\includegraphics[width=\linewidth,height=4in]{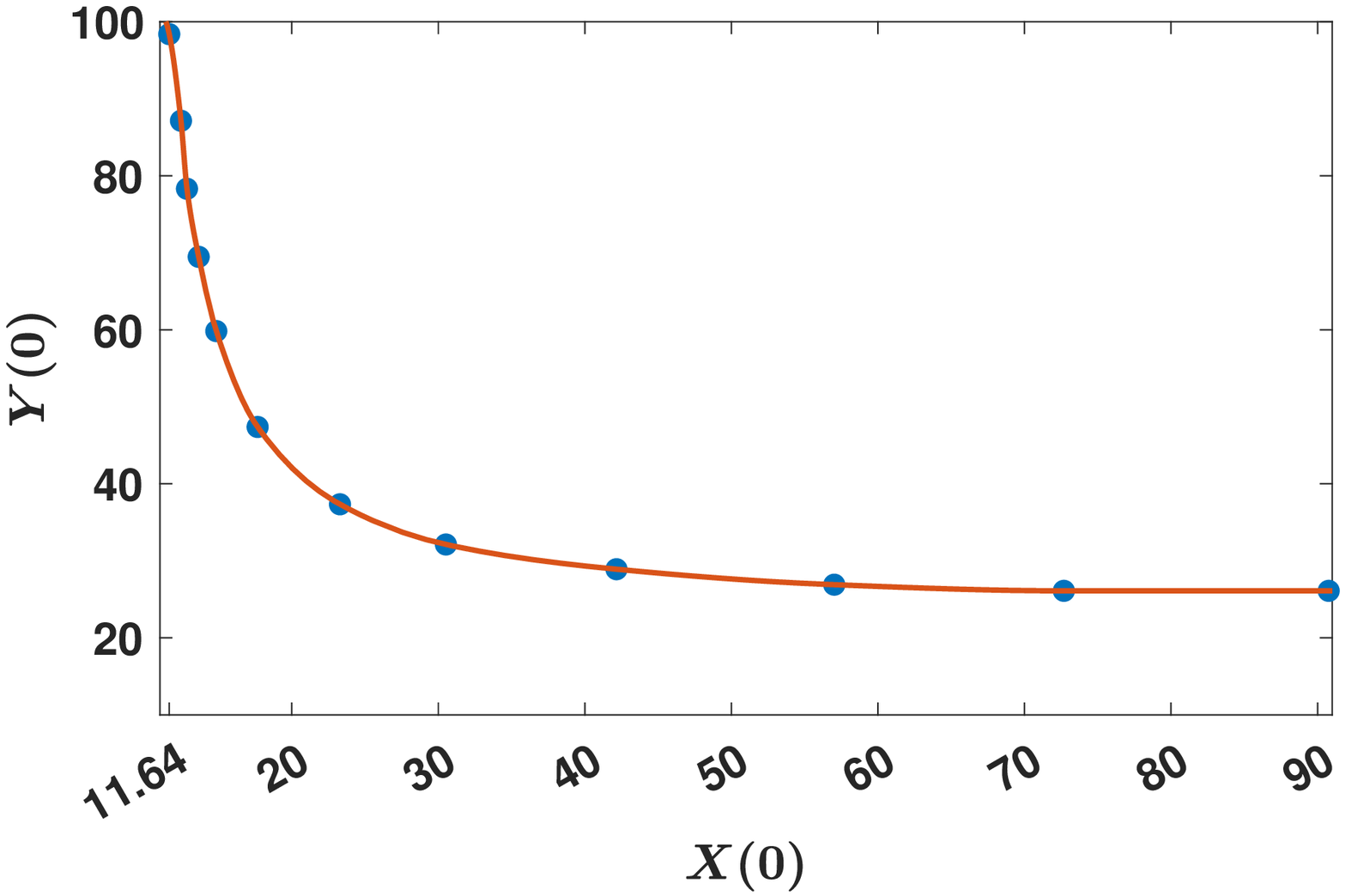}}}
   \caption{(A) Dynamics of non-delayed predator-prey dynamics with feedback control given by the mathematical model in \eqref{f1model} when certain initial conditions are applied. The colored regions indicate two possible outcomes: the red region represents a non-blowup case, while the blue region represents a blow-up case if we choose initial conditions from that region. The simulation used specific parameter values: $M=1.2$, $p=2$, $D=0.4$, $E=0.2$, $A=0.2, u =.02, \tau=0$ and $C=0.3$ (B) Shape-preserving interpolation of the boundary of the region of attraction. The blow-up tolerance parameter is assumed as $10^8.$}
	\label{fig:blowup_non_delay_f1_domain}
\end{figure}

\begin{figure}[ht]
	{\scalebox{0.49}[0.49]{
			\includegraphics[width=\linewidth,height=4in]{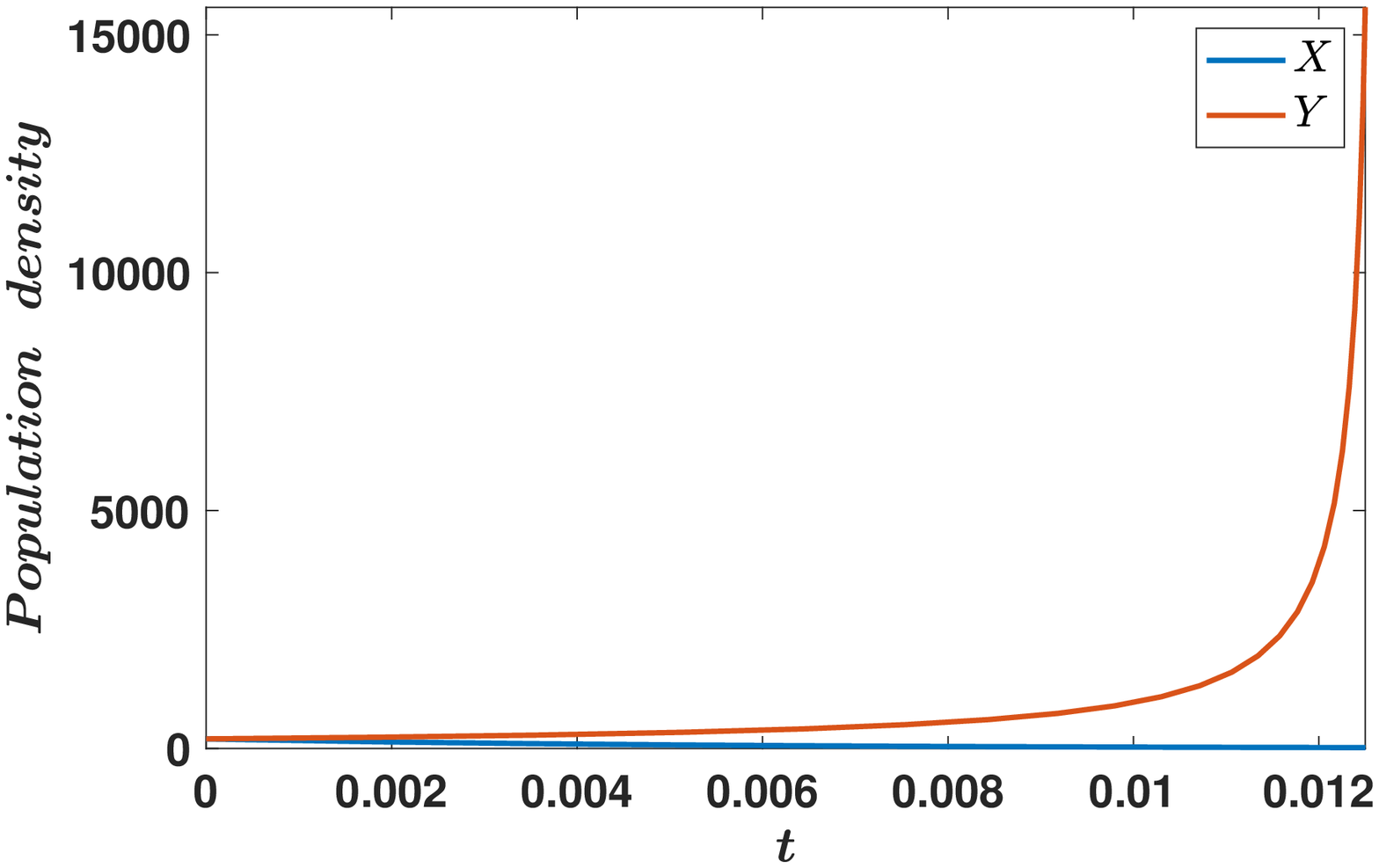}}}
	{\scalebox{0.49}[0.49]{
			\includegraphics[width=\linewidth,height=4in]{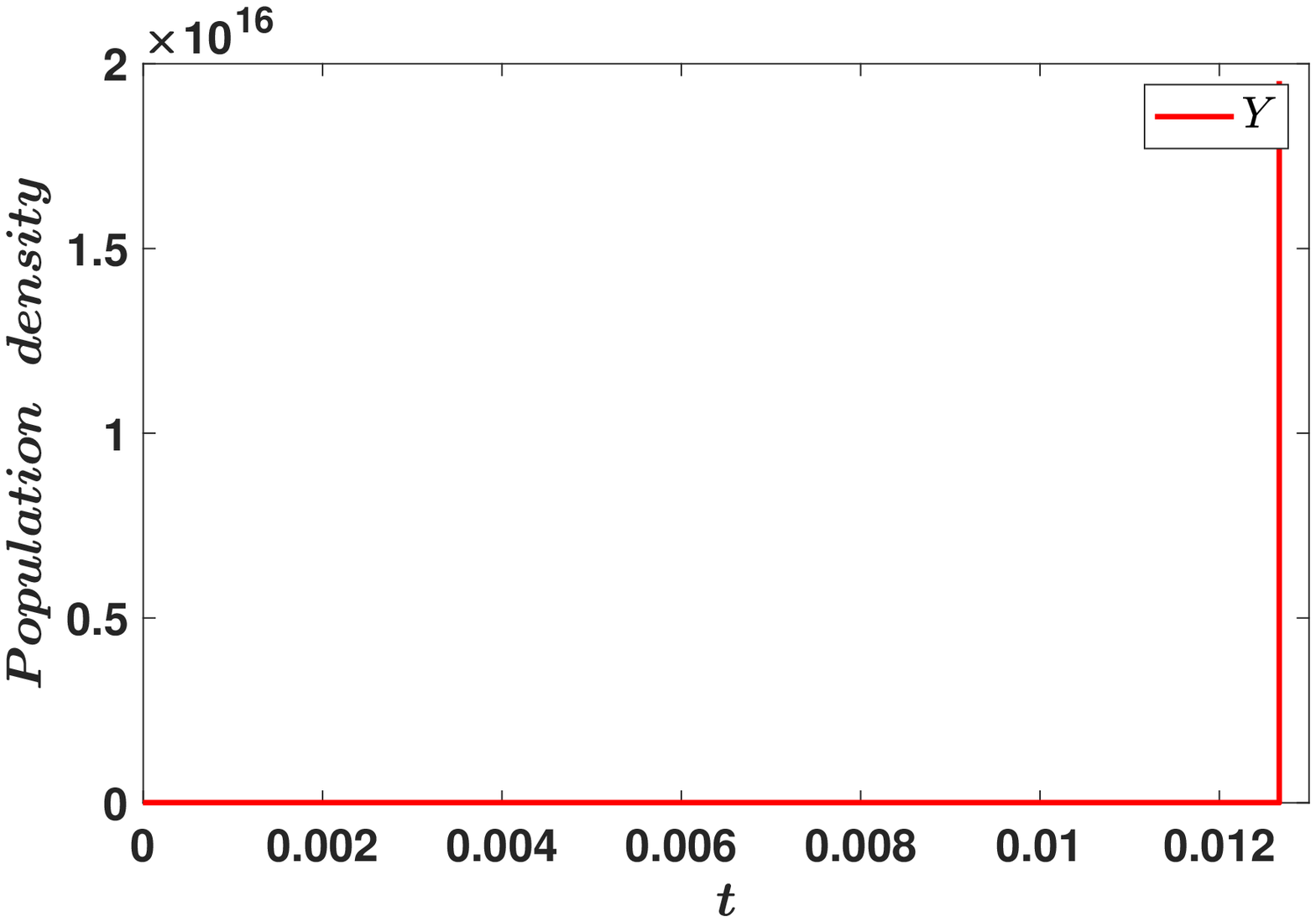}}}
		\caption{Time series simulation of the delayed predator-prey dynamics given by the mathematical model in \eqref{f1model}. The simulation used specific parameter values: $M=1.2$, $p=2$, $D=0.4$, $E=0.2$, $A=0.2, u =.06, \tau=2$ and $C=0.3$,, with initial data $[X(0),Y(0)]=[200,200]$.}
	\label{fig:blowup_delay_f1}
\end{figure}

\section{The blow-up boundary}

In this section we ascertain the curve in the phase space that delineates initial conditions that lead to blow-up in finite time versus those that are attracted to equilibrium or limit cycles, hence stay bounded - this is the \emph{blow-up boundary}.

WLOG assumes $R=K=1$. To prove that system \eqref{model} and \eqref{delay_model} blow up in finite time, we will use the blow-up techniques mentioned in \cite{rana2017,upadhyay2016}.
\begin{conjecture}
\label{ode_blowup}
    Consider the system \eqref{model}, for any choice of parameter such that $\delta_1>0$, and the initial condition satisfies the largeness conditions
    \begin{align}\label{large}
       \ln \Big( \dfrac{|X(0)|}{\frac{E}{D-\delta_1} - A} \Big)  > \frac{1}{\delta_1 |Y_0|},
    \end{align}
    then the predator population will explode to infinity in finite time, i.e.,
    \[ \lim_{t \to T^* < \infty} ||Y|| \to \infty,\]
    where the blow-up time is given by
    \[T^* \le \frac{1}{\delta_1 |Y_0|}\]
\end{conjecture}

Let's analyze this conjecture by considering a few cases:

\textbf{Case I:} Assume $D>\frac{E}{A}$. Hence the result follows trivially, since then the predator nullcline would be of the form
\[ \displaystyle\frac{d Y}{d t} \ge \widehat{C} Y^2, \]
where $\widehat{C}>0$, and hence the predator blows up to infinity in finite time.

\textbf{Case II} Assume  $ \underset{X(t)}{\min} \{ D-\frac{E}{X+A}\} = \delta_1 >0$, then again we have the finite time blow-up for predator population since
\[ \displaystyle\frac{d Y}{d t} \ge \delta_1 Y^2. \]

\textbf{Case II:} Let's work on non-obivous case, i.e., when $D-\frac{E}{X+A}>\delta_1>0$. For this case, we need to make sure the $D- \frac{E}{X+A} $ has the same sign, i.e.,

\begin{align*}
   D- \frac{E}{X+A} >\delta_1 >0 \quad \forall \hspace{.05in} X
\end{align*}

 Equivalently,
 \begin{align}\label{sign}
   X > \frac{E}{D - \delta_1} - A >0
\end{align}

Let's consider the prey nullcline, and using the positivity and logistic comparison arguments, we have $|X(t)|\le 1 \hspace{.05in} \forall \hspace{.05in} t$. On using these estimates in prey nullclines, we have


\[ \displaystyle\frac{d X}{d t}=X\left(1-\displaystyle X \right) - \dfrac{M X Y}{X^p + C}  \quad \implies \quad X \ge X(0) e^{-t}. \]
Using the lower bound for prey population $X$ in \eqref{sign}, we have

\begin{align*}
    |X(0)|e^{-t} > \frac{E}{D - \delta_1} - A \quad
    \implies \quad  \ln \Big( \dfrac{|X(0)|}{\frac{E}{D-\delta_1} - A} \Big) >t
\end{align*}

Moreover, we know that: 

\[ \displaystyle\frac{d Y}{d t} = \delta_1 Y^2, \]

blows up in $t=T^*\le \frac{1}{\delta_1 |Y_0|}.$ Hence, we have

\begin{align*}
    \ln \Big( \dfrac{|X(0)|}{\frac{E}{D-\delta_1} - A} \Big) > \frac{1}{\delta_1 |Y_0|} \quad \implies \quad |Y_0| \ln \Big( \dfrac{|X(0)|}{\frac{E}{D-\delta_1} - A} \Big) > \dfrac{1}{\delta_1}. 
\end{align*}

Therefore under the largeness restriction on the initial data given by \eqref{large}, the predator population will blow up to infinity in finite time.

\begin{conjecture}
\label{con:delay_ode_blowup}
    Consider the system (\ref{delay_model}), for any choice of a parameter such that $\delta_1>0$, and the initial condition satisfies the largeness conditions given by \eqref{large}, then the predator population will explode to infinity in finite time.
\end{conjecture}

    Consider the prey nullcline defined for the delayed model \eqref{delay_model}, and make use of the logistics comparison argument, which entails that $X(t-\tau) \le 1 \hspace{.05in} \forall \hspace{.05in} t$. Hence, we have
    \begin{align*}
        \displaystyle\frac{d X}{d t} &=X\left(1-\displaystyle X(t-\tau) \right) - \dfrac{M X Y}{X^p + C}  \ge -X(t) \quad \\
        \implies \quad X &\ge X(0) e^{-t} \quad \forall \hspace{.05in} t 
    \end{align*}
The lower-bound on prey population gives the following inequality,
\begin{align*}
     - \dfrac{1}{X+A} & \ge -\dfrac{1}{X(0) e^{-t} + A} 
\end{align*}
On substituting this in predator nullclines, we have
\begin{align*}
     \displaystyle\frac{d Y}{d t} &\ge \Big( D - \dfrac{1}{X(0) e^{-t} + A } \Big) Y^2
\end{align*}
So, if we have

\[ D - \dfrac{1}{X(0) e^{-t} + A} >\delta_1 >0 \quad \implies \Big( \dfrac{|X(0)|}{\frac{E}{D-\delta_1} - A} \Big) >t. \]

Moreover, we know that:

\[ \displaystyle\frac{d Y}{d t} = \delta_1 Y^2, \]

blows up in $t=T^*\le \frac{1}{\delta_1 |Y_0|}.$

\begin{align*}
    \ln \Big( \dfrac{|X(0)|}{\frac{E}{D-\delta_1} - A} \Big) > \frac{1}{\delta_1 |Y_0|} \quad \implies \quad |Y_0| \ln \Big( \dfrac{|X(0)|}{\frac{E}{D-\delta_1} - A} \Big) > \dfrac{1}{\delta_1}. 
\end{align*}

Given the condition on the initial data as stated in \eqref{large}, the predator population will experience blow-up, meaning it will grow to infinity within a finite time.

\begin{figure}[ht]
    {\scalebox{0.49}[0.49]{
			\includegraphics[width=\linewidth,height=4in]{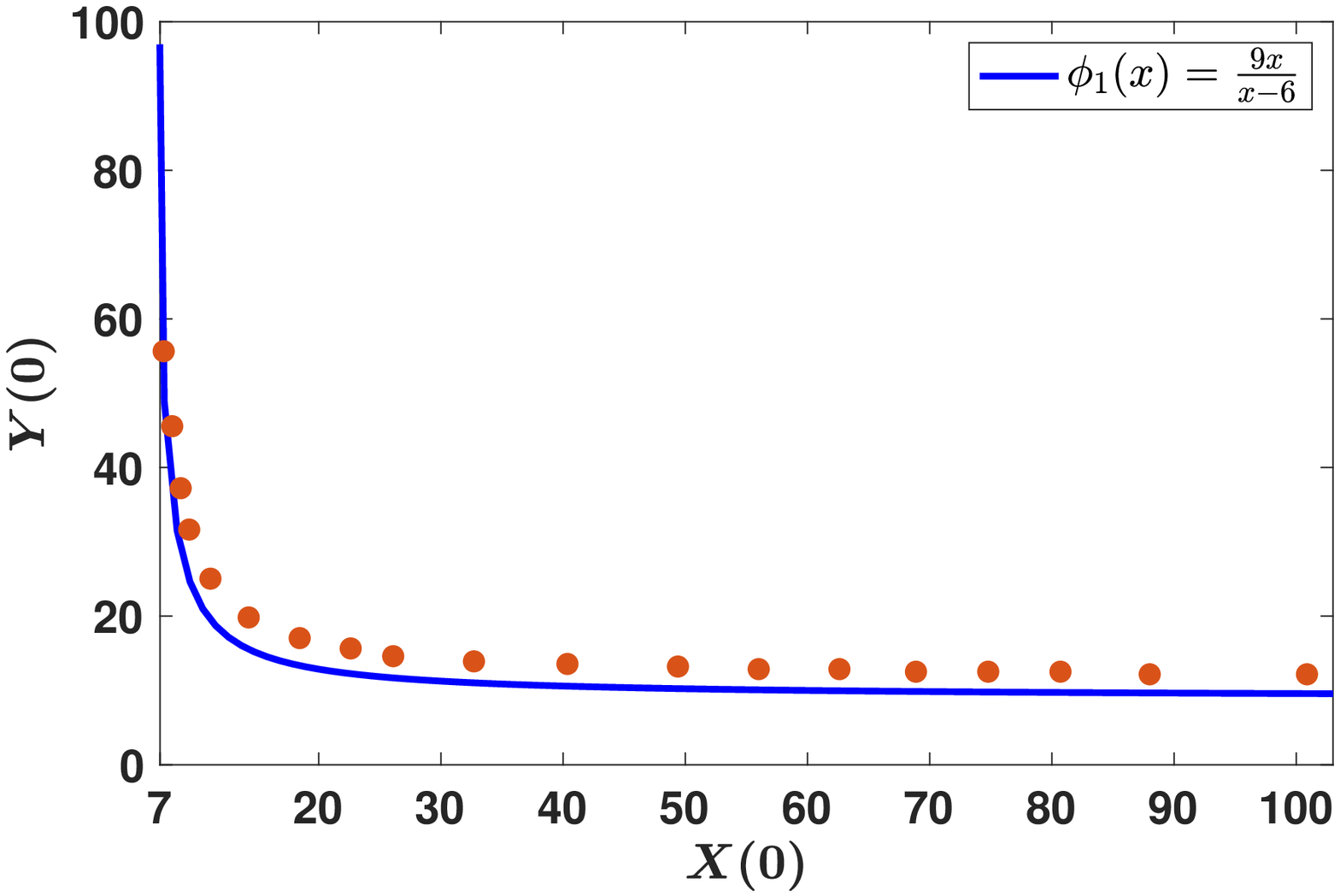}}}
	{\scalebox{0.49}[0.49]{
		\includegraphics[width=\linewidth,height=4in]{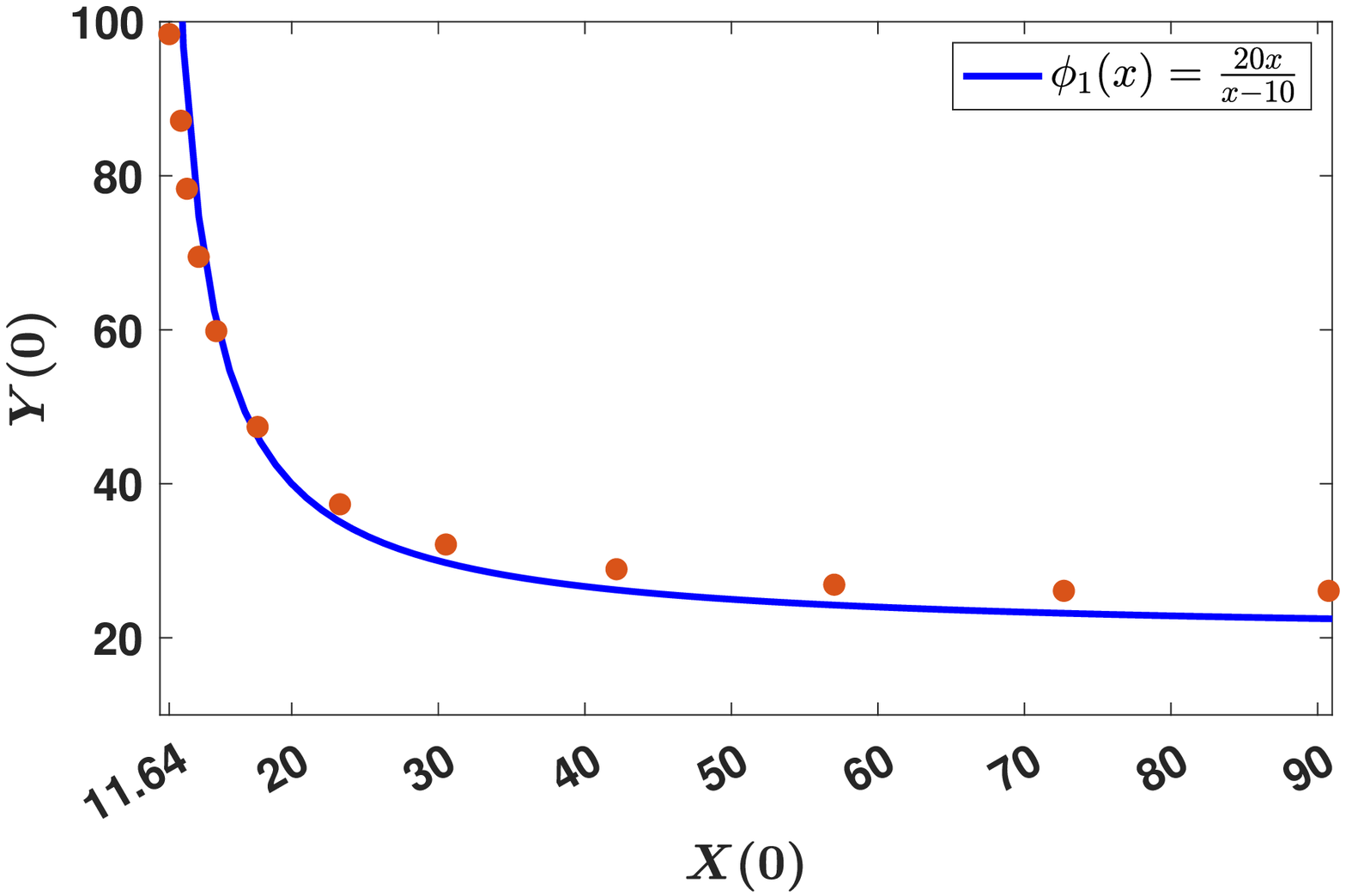}}}
    \caption{Approximation of blow-up boundary with the help of nonlinear function $\phi_1 (x)$ for non-delayed predator-prey dynamics with and without feedback control. The red dots in the figure represent points on the blow-up boundary. In panel (A), the parameters for the non-delayed model without feedback control are the same as those used in Figure \ref{fig:blowup_ic}. Panel (B) parameters are identical to those in Figure \ref{fig:blowup_non_delay_f1_domain}.}

	\label{fig:blowup_phi}
\end{figure}

\section{The basin of attraction}

The basin of attraction for the model under consideration is fairly large. That is initial data much larger than carrying capacity can be attracted to equilibrium. However, despite the non-delayed model being a two-species one (thus chaos is precluded) the structure of the basin of attraction can be richer than a unique interior equilibrium or limit cycle dynamics. We provide a few details next.

We first provide a summary for the stability of the interior equilibrium point in the following theorem.

\begin{theorem}\label{thm:stable_int}
The interior equilibrium of model \eqref{model} is locally asymptotically if $C>\frac{\left(\frac{E}{D}-A\right)^p (A D p+A D+D K p-E p-E)}{E-A D}$.
\end{theorem}

\begin{proof}
The proof follows directly from the linear stability analysis in the appendix.  
\end{proof}

Next, we discuss the critical point of the prey's protection by the environment ($C$) around the interior equilibrium point where changes in stability occur.\\

\begin{theorem}\label{thm:hopf-bif}
The model \eqref{model} experiences a Hopf bifucation around the interior equilibrium point $E_2$ when $C$ crosses some critical value  $C_H$ where $C=C_H=\frac{\left(\frac{E}{D}-A\right)^p (A D p+A D+D K p-E p-E)}{E-A D}$.
\end{theorem}

\begin{proof}
The proof is routine and thus omitted for brevity.  
\end{proof} 

We shall provide an example to corroborate with the explicit parameter dependent critical value $C_H$ from Theorem \ref{thm:hopf-bif}. \\

\begin{example}\label{exam1}
 Consider the parameters $M=1.2,p=2,D=0.4,E=0.2$, and $A=0.2$ as used in \cite{patra2022effect}. Then    
 \begin{align*}
     C_H &=\frac{\left(\frac{E}{D}-A\right)^p (A D p+A D+D K p-E p-E)}{E-A D}\\
     &=\dfrac{(0.5-0.2)^2(0.16+0.08+0.8-0.4-0.2)}{0.2-0.08} \\
     &=\dfrac{0.0393}{0.12}\\
     &=0.33
 \end{align*}
 The value obtained above corroborates with the $C_H$ value computed with a computer algebra system as depicted in \cite{patra2022effect} (Figure 5(a) on page 13).
\end{example}

From this it appears if the interior equilibrium is a nodal sink its basin of attraction can be fairly large. However if it is a spiral sink, this may not be so. The interior equilibrium could loose stability via a Hopf bifurcation as seen in Theorem \ref{thm:hopf-bif} to yield a limit cycle. Thus fairly large initial conditions could now be attracted to this limit cycle. The situation can be more interesting and even richer dynamics could abound. We numerically see the possibilities of multiple concurrent limit cycles. This was also observed in \cite{patra2022effect}. Thus we  conjecture that the occurrence of two concurrent limit cycles is possible,

\begin{conjecture}\label{con3}
    Consider \eqref{model}. Then there exist single parameter regimes, under which the interior equilibrium point can bifurcate to produce at least two limit cycles.
\end{conjecture}

This is also possible via varying two parameters.

\begin{conjecture}\label{con4}
    Consider \eqref{model}. Then there exist two parameter regimes, under which a Bautin bifurcation is possible.
\end{conjecture}

Further these could collide and annihilate each other,

\begin{conjecture}\label{con5}
    Consider \eqref{model}. Then there exist parameter regimes, under which a cyclical fold bifurcation is possible.
\end{conjecture}

\begin{figure}[ht]
{\scalebox{0.49}[0.49]{
			\includegraphics[width=\linewidth,height=4in]{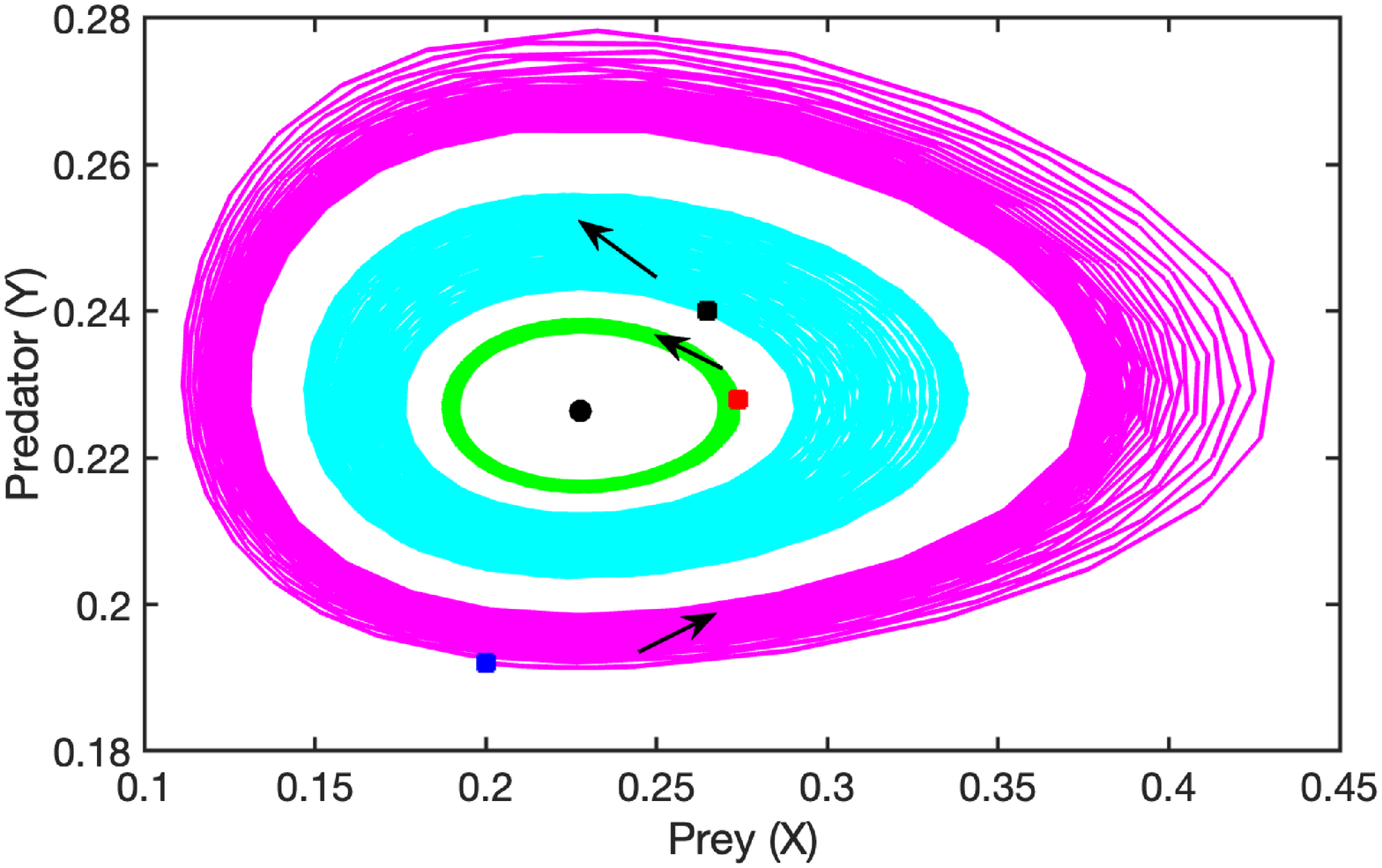}}}
	{\scalebox{0.49}[0.49]{
			\includegraphics[width=\linewidth,height=4in]{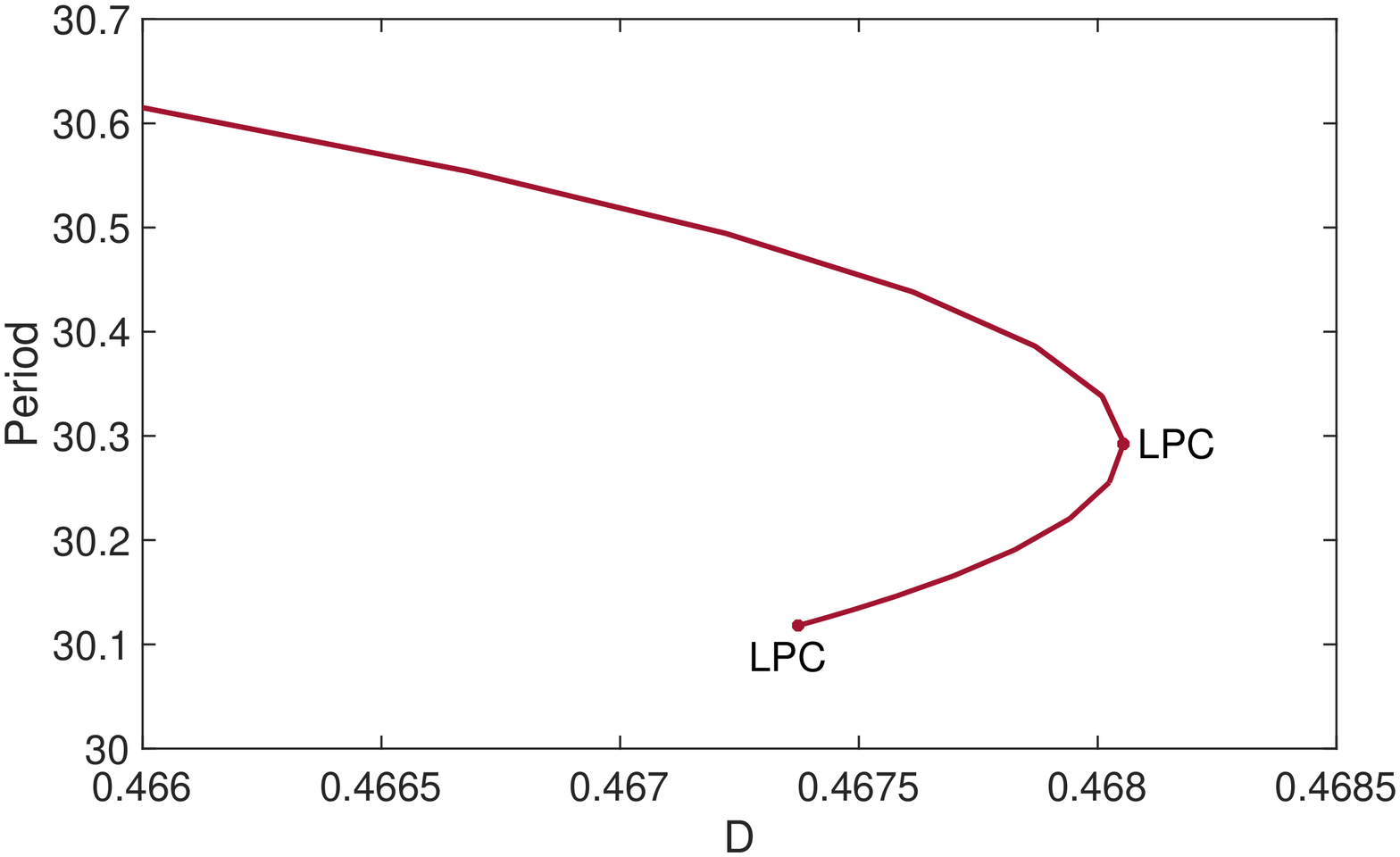}}}
    \caption{The presence of two limit cycles is illustrated in the figures. Left Figure: The stable interior equilibrium point $(0.2277,0.2264)$ is denoted by a solid black circle, while the initial conditions $(0.274,0.228)$, $(0.265,0.24)$, and $(0.20,0.192)$ are represented by solid red, black, and blue squares, respectively. 
    Right Figure: Period of cycles versus the parameter $D$. Limit Point Cycle (LPC) observed at $D=0.4674$ and $D=0.4681$.
    The chosen parameter values for this illustration are as follows: $R=1$, $K=1$, $M=1.2$, $p=2$, $D=0.4676$, $E=0.2$, $A=0.2$, and $C=0.3$.}
	\label{fig:2limit-cycles}
\end{figure}

\begin{figure}[H]
{\scalebox{0.49}[0.49]{
			\includegraphics[width=\linewidth,height=4in]{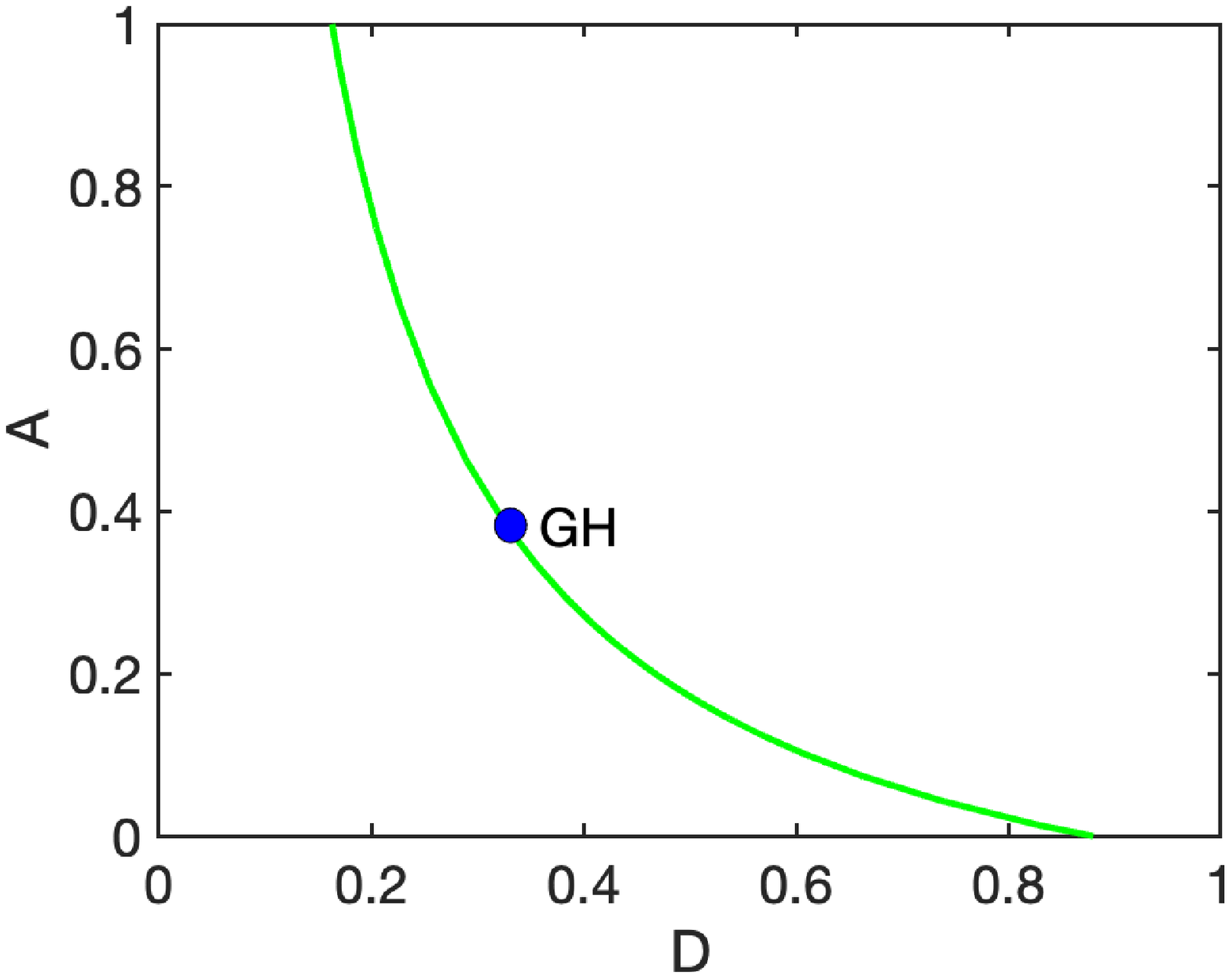}}}
	{\scalebox{0.49}[0.49]{
			\includegraphics[width=\linewidth,height=4in]{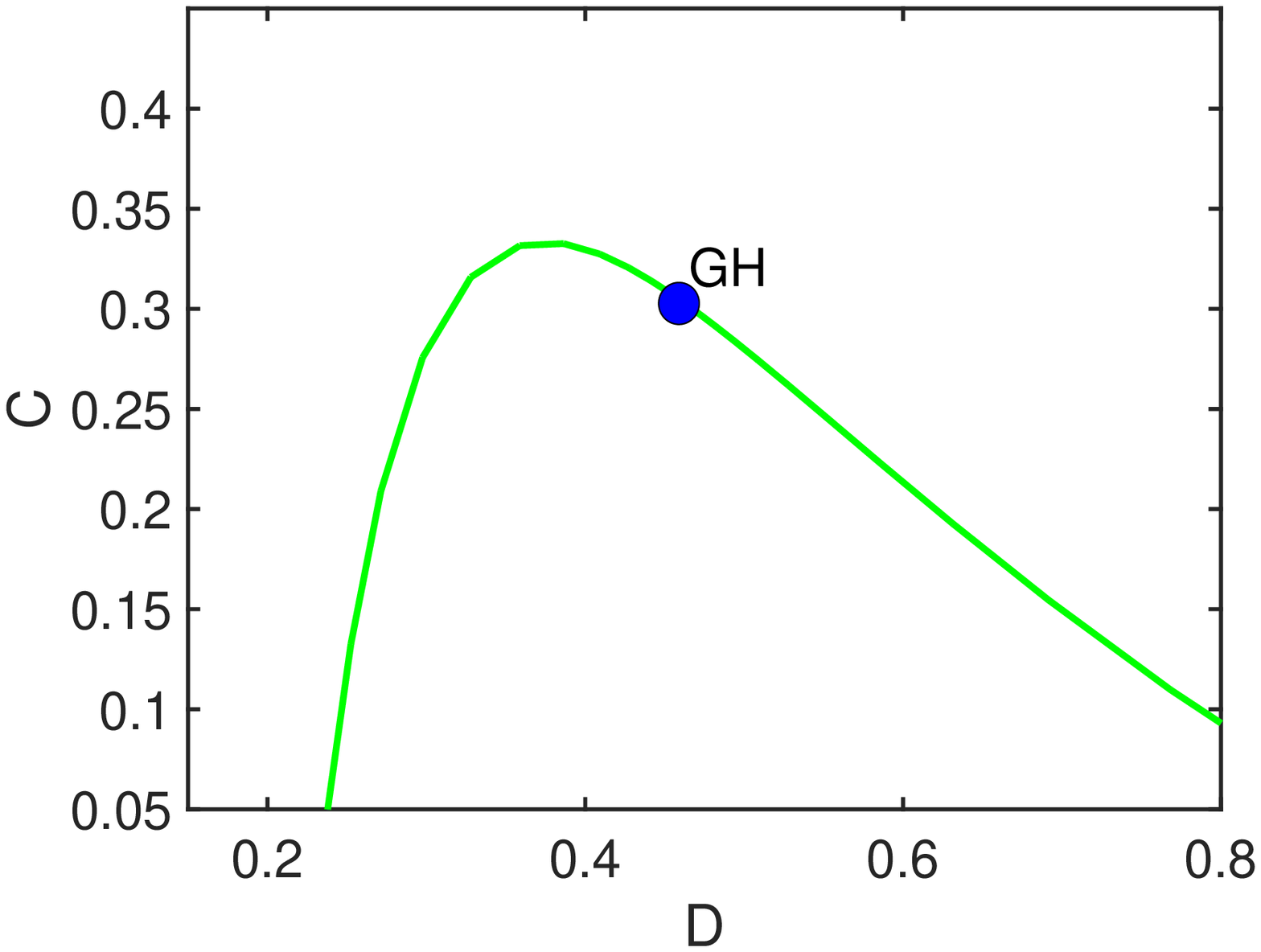}}}
   \centering
{\scalebox{0.49}[0.49]{
			\includegraphics[width=\linewidth,height=4in]{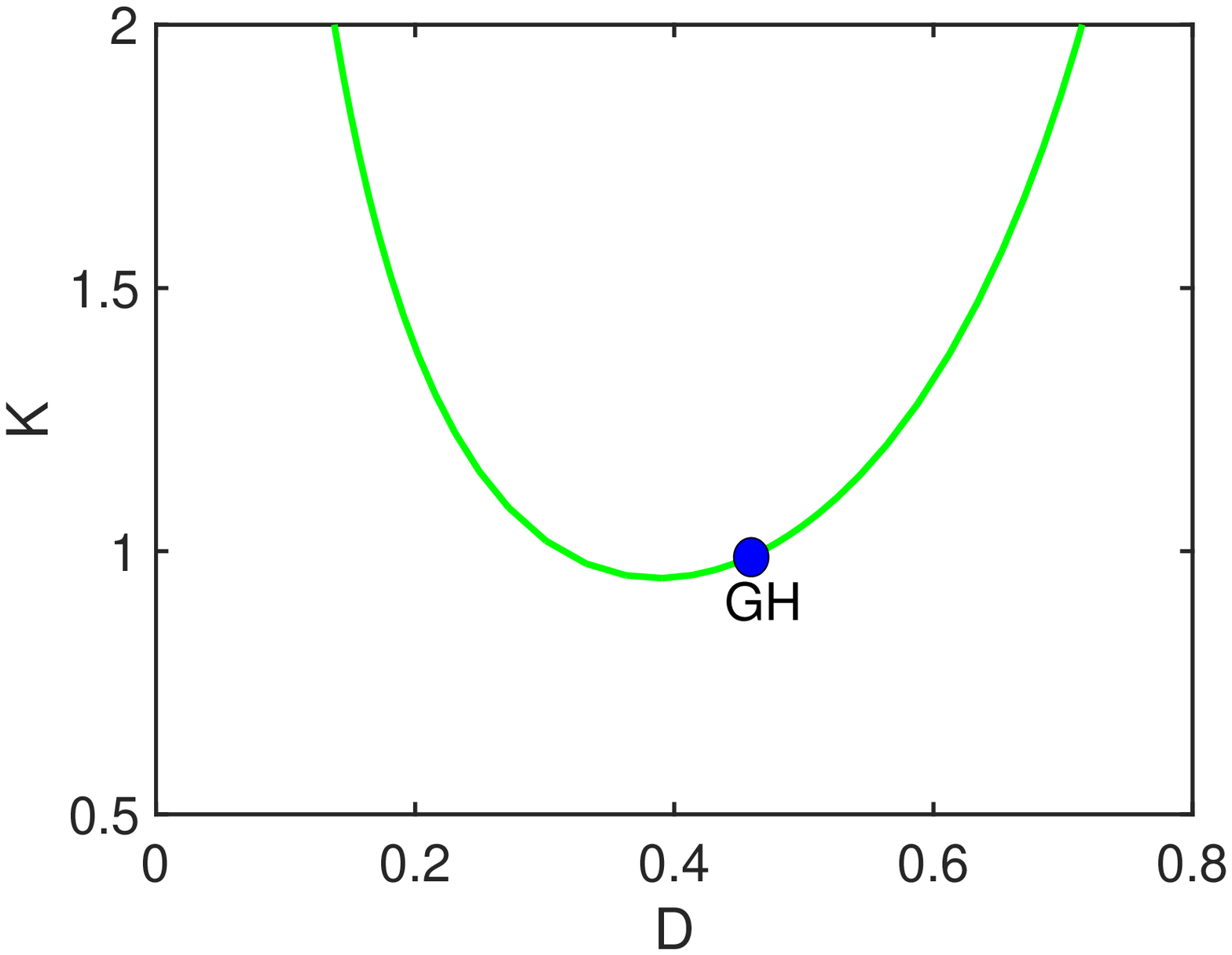}}}   
    \caption{Two parameter bifurcation plots showing Bautin bifurcation. Left Figure: $(D,A)=(0.3302,0.3777)$. Right Figure: $(D,C)=(0.4576,0.3055)$. Bottom Figure:  $(D,K)=(0.4699,0.9910)$. The chosen parameter values for this illustration are as follows: $R=1$, $K=1$, $M=1.2$, $p=2$, $E=0.2$, $A=0.2$, and $C=0.3$.}
	\label{fig:generalized-hopf}
\end{figure}

\begin{Remark}\label{remark:LPC}
In Figure \ref{fig:2limit-cycles}, the initial conditions emanating from the solid blue and black squares exhibit convergence towards an outer stable limit cycle, whereas the initial conditions from black and red squares display movement away from an inner unstable limit cycle. To visualize the limit point cycle (LPC) in an alternate way, we show a curve with limit points clearly indicating the existence of two limit cycles for $D<0.4681$. Additionally, the normal form coefficient of the LPC derived with the aid of MATCONT is $-4.01\times 10^{-2}\neq 0$, thus the limit cycle has a fold.
\end{Remark}

In these situations the basin structure would be richer. The basin for the interior (stable) equilibrium would be small, and larger data would be attracted to the outer limit cycle. Thus bi-stable behavior, between a stable equilibrium and a stable limit cycle is possible, with very large data leading to blow-up solutions.

\section{Numerical Simulations}
We performed a time series simulation of the mathematical model given by \eqref{model} and \eqref{delay_model} using specific parameter values: $R=1$, $K=1$, $M=1.2$, $p=2$, $D=0.5$, $E=0.2$, $A=0.2$, and $C=0.3$.

For the non-delay model \eqref{model}, with initial data of $[X(0),Y(0)]=[78,30]$, we observed that the predator population experiences blow-up to infinity in a finite time of $T^* \approx 6.789603\times 10^{-2}$. We ensured that the chosen parameter values satisfied the requirements for a blow-up, as stipulated by Theorem \ref{ode_blowup}. The numerical time series simulation was generated using \verb|MATLAB|'s standard solver for ordinary differential equations, ode45.

For the delayed model with the same parameter values as the non-delay model and an additional delay parameter $\tau=1$, we found that the predator population explodes to infinity in a finite time of $T^* \approx 2.029316\times 10^{-2}$ for constant positive initial data of $[100,100]$ (See Figure \eqref{fig:blowup_delay}). We ensured that the selected parameter values met the requirements for a blow-up, as prescribed by Theorem \ref{thm2dm}. The time series simulation was produced using dde23, which is \verb|MATLAB|'s solver for delay differential equations (DDEs) with constant delays.

We conducted numerical investigations to analyze the region of attraction or where the solution is bounded. Figure~\ref{fig:blowup_ic} displays the results, where colored regions indicate two possible outcomes. The red region represents a non-blowup case, whereas the blue region represents a blow-up case if we choose initial conditions from respective regions. The simulations were conducted using the non-delayed model \eqref{model} with specific parameter values: $R=1$, $K=1$, $M=1.2$, $p=2$, $D=0.5$, $E=0.2$, $A=0.2$, and $C=0.3$. We performed the simulations $9\times 10^4$ times, picking an initial data point from a mesh grid where $0\le X(0), Y(0)\le 100$.

\section{Conclusion}

The current manuscript aims to provide commentary on the dynamics of the predator-prey model introduced in \cite{patra2022effect}. The authors establish several stability and boundedness results for both the delayed and non-delayed models under certain parameter restrictions. However, our findings reveal that the solutions to \eqref{model} are actually unbounded and can even exhibit finite-time blow-up for sufficiently large data, contrary to the specified parameter restrictions of Theorem 3.2 \cite{patra2022effect}. We demonstrate this through Theorem \ref{thm:thm2}. Consequently, the positive quadrant $\mathbb{R}^{2}$ is not invariant for system \eqref{model}, as shown in Corollary \ref{cor:bdd_patra11}.

In addition, the authors introduce a delay in the predator equation, motivated by the gestation period, which ensures global in-time existence, as demonstrated in Theorem \ref{delay_ode_blowup}.

Furthermore, we introduce the concept of quenching in finite time for the prey species. Subsequently, we propose that the predator's blow-up time and the prey's quenching time are equivalent. For the non-delayed model, this is presented in Theorem \ref{conj:quench1} and validated numerically in Remark \ref{rem:quen1} and Figure \ref{fig:quench}. Similarly, for the delayed model, this is shown via Theorem \ref{conj:quench2} and numerically validated in Remark \ref{rem:quen2} and Figure \ref{fig:quench_delay}. To validate our analytical findings, we conduct several numerical simulations, as depicted in Figs [\ref{fig:blowup},\ref{fig:blowup_ic},\ref{fig:quench},\ref{fig:blowup_delay},\ref{fig:quench_delay},\ref{fig:blowup_non_delay_f1},\ref{fig:blowup_non_delay_f1_domain},\ref{fig:blowup_delay_f1},\ref{fig:blowup_phi}]. To the best of our knowledge, this is the first instance of \emph{simultaneous} blow-up and quenching reported in the predator-prey literature. It is also one of the first instances where a smooth RHS can lead to quenching. To this end the finite time blow-up dynamic of the predator is exploited. The idea of using blow-up solutions to mimic the explosive increase of invasive species has been considered in the literature \cite{PQB16}. This has been seen to fit population patterns where the explosive increase of invasive Burmese Pythons in the Florida Everglades, has resulted in the sharp decline of mammalian prey \cite{DM12}. To this end the simultaneous blow-up and quenching models introduced in the current work may be more apt, as these would not only catch the sharp population increase in predator populations, but the sharp decrease in their prey populations as well. Furthermore, as quenching requires blow-up in the time derivative (while the state variable stays bounded and positive) such a model could depict the sharp decline of prey to very low (but still positive) levels, such that they fall below almost any sensitive (and endangered) levels. To this end fitting some of these data sets via the current models would make for an interesting future endeavor.

Furthermore, we provide numerical evidence that the linear feedback control proposed in \cite{patra2022effect} is inadequate in preventing blow-ups for both delayed and non-delayed models. This is demonstrated numerically in Remarks \ref{rem:feed1} and \ref{rem:feed2}, as well as Figs [\ref{fig:blowup_non_delay_f1},\ref{fig:blowup_delay_f1}]. In \cite{patra2022effect}, the authors proposed and proved the asymptotic stability of the interior solution for the delayed linear feedback control model under specific parameter restrictions. However, we disprove their claim by providing a counterexample, as illustrated numerically in Remark \ref{rem:feed2} and Figure \ref{fig:blowup_delay_f1}.

The discussion on the basin of attraction is comprehensive, incorporating the support of various theorems (refer to Theorems \ref{thm:stable_int} and \ref{thm:hopf-bif}). To provide a practical understanding, we present an illustrative example (refer to Example \ref{exam1}). In \cite{patra2022effect}, the presence of multiple concurrent limit cycles was initially observed by the authors. Expanding upon their findings, we present several conjectures (refer to Conjectures \ref{con3}, \ref{con4}, and \ref{con5}). To validate these conjectures numerically, we conducted simulations and obtained compelling results (refer to Remark \ref{remark:LPC}, and Figures \ref{fig:2limit-cycles} and \ref{fig:generalized-hopf}). These simulations involved systematic variations of either one or two parameters, providing evidence for the occurrence of two concurrent limit cycles. Note, the occurrence of multiple limit cycles for both Gauss type as well as Holling-Tanner type predator prey models has been shown in the literature \cite{HS89,SG99}. However, these models all possess globally existing bounded solutions, for all positive initial conditions and across all parameter regimes. However, this is not the case in the current class of models, where the top predator is modeled via a modified Leslie-Gower scheme. In this setting we know via Theorem \ref{thm:thm2} that finite time blow-up is possible. Thus the dynamics of the system are richer than the basic Gauss type as well as Holling-Tanner type predator prey models. Herein one could it would be pertinent to ask the maximum radius of the outer (stable) limit cycle. 

Ecologically speaking one could see bi-stability between the interior equilibrium and the stable outer limit cycle, see Fig. \ref{fig:2limit-cycles}, where the predator and prey populations would go between stable levels to cyclical populations. Note further one could see blow-up solutions if initial conditions were very far from the outer cycle. To this end one could conjecture a the possible existence of a third larger cycle encircling the outer stable one - and thus possible changes to the structure of the blow-up boundary seen in Fig. \ref{fig:blowup_ic}. Currently the blow-up boundary is approximated numerically, see Fig \ref{fig:blowup_phi}, using a function of the form $\phi(x) = \frac{ax}{bx-c}$, which follows from estimates in Theorem \ref{thm:thm2}. Later analysis shows fits of the form $\phi(x) = \frac{1}{b \log(cx)}$ might also be feasible. The exact sufficient (and necessary) conditions to compute the blow-up boundary for these class of model systems remains open.

Overall, one must use extreme caution in deriving global in time existence results, in predator-prey systems when the top predator is modeled via the modified Leslie-Gower scheme. It is more apt to derive restrictions on the initial data such that solutions remain bounded. Also, one can look at the various ecological mechanisms that could prevent blow-up solutions, for initial conditions that would indeed lead to blow-up without these mechanisms.

\section*{Conflict of interest}
\noindent The authors declare there is no conflict of interest in this paper.


\appendix
\section{Stability Analysis of Non-Delayed Model}
In this section, we shall investigate the existence of biologically feasible equilibria of the model \eqref{model}. Also, we will analyze the dynamics of the model around these equilibria. 

\subsection{Existence of Equilibrium Points}
From model \eqref{model}, there are three biologically feasible equilibria, namely
\begin{itemize}
    \item[(i)] The extinction (or population-free) equilibrium point is given by ${E_0}=(0,0)$.
    \item[(ii)] The predator-free (or axial) equilibrium point is given by ${E_1}=(K,0)$.
    \item[(iii)] The unique interior equilibrium (or coexistence) point is given by
\[{E_2}=(X^*,Y^*)=\left(\frac{E-A D}{D},-\frac{\frac{C R (E-A D)}{D}-K R \left(\frac{E-A D}{D}\right)^p+R \left(\frac{E-A D}{D}\right)^{p+1}-C K R}{K M}\right),\]
where $X^*$ and $Y^*$ are the positive solutions of the prey and predator nullclines:
\begin{align*}
0&=R \left(1-\displaystyle\frac{X}{K}\right) - \dfrac{M  Y}{X^p + C}\vspace{2ex}\\
0&=\left( D- \dfrac{E}{X + A} \right).
\end{align*}
\end{itemize}

\subsection{Local Stability Guidelines}
In this subsection, we shall investigate the possible parametric restrictions for the local stability of the model dynamics around each equilibrium point with the aid of the Jacobian matrix.
The Jacobian matrix is given by
\[
J=\left[
\begin{array}{cc}
 \frac{M Y \left((p-1) X^p-C\right)}{\left(C+X^p\right)^2}-\frac{2 R X}{K}+R & -\frac{M X}{C+X^p} \\
 \frac{E Y^2}{(A+X)^2} & 2 Y \left(D-\frac{E}{A+X}\right) \\
\end{array}
\right].
\]

Next, we will investigate the possible stability criteria around each equilibrium point.

\subsubsection{Extinction Equilibrium}
The Jacobian matrix evaluated at the extinction equilibrium ($J_{E_0}$) becomes

\[
J_{E_0}=\left[
\begin{array}{cc}
R &   0\\
 0 &  0 \\
\end{array}
\right],
\] 

where

\begin{equation}\label{eqn: extinct-dettr}
    \det (J_{E_0})=0 \quad \text{and} \quad \tr (J_{E_0})=R.
\end{equation}

\noindent Clearly, the extinction equilibrium point is a saddle based on equation \eqref{eqn: extinct-dettr}.

\subsubsection{Predator-free Equilibrium}
The Jacobian matrix evaluated at the predator-free equilibrium ($J_{E_1}$) becomes

\[
J_{E_1}=\left[
\begin{array}{cc}
-R &   -\frac{M K}{C+K^p}\\
 0 &  0 \\
\end{array}
\right],
\] 

where

\begin{equation}\label{eqn: pred-free-dettr}
    \det (J_{E_1})=0 \quad \text{and} \quad \tr (J_{E_1})=-R.
\end{equation}

\noindent The stability of the predator-free equilibrium cannot be determined from the information given in equation \eqref{eqn: pred-free-dettr} via linearization. Here, the equilibrium point is said to be degenerate. The stability of the predator-free equilibrium can be determined by using the center manifold theory.

\subsubsection{Interior Equilibrium}
The Jacobian matrix evaluated at the interior equilibrium ($J_{E_2}$) becomes

\[
J_{E_2}=\left[
\begin{array}{cc}
  J^*_{11} &   J^*_{12}\\
 J^*_{21} &  J^*_{22} \\
\end{array}
\right],
\]  

where 
\begin{align*}
    J^*_{11}&= \frac{C R (A D-E)+R \left(\frac{E}{D}-A\right)^p (A D (p+1)+D K p-E (p+1))}{D K \left(\left(\frac{E}{D}-A\right)^p+C\right)} \\
    J^*_{12}&= -\frac{M (E-A D)}{D \left(\left(\frac{E}{D}-A\right)^p+C\right)}<0\\
    J^*_{21}&=\frac{R^2 (E-D (A+K))^2 \left(\left(\frac{E}{D}-A\right)^p+C\right)^2}{E K^2 M^2}>0 \\
    J^*_{22}&=0. \\
\end{align*}
Now, we note that
\begin{equation}\label{eqn:det}
\det J_{E_2}=-J^*_{12}J^*_{21}=\frac{R^2 (E-A D) (A D+D K-E)^2 \left(\left(\frac{E}{D}-A\right)^p+C\right)}{D E K^2 M}>0
\end{equation}

and

\begin{equation}
\tr J_{E_2}=J^*_{11}= \frac{C R (A D-E)+R \left(\frac{E}{D}-A\right)^p (A D (p+1)+D K p-E (p+1))}{D K \left(\left(\frac{E}{D}-A\right)^p+C\right)}. 
\end{equation}

\noindent The determinant of the interior equilibrium is always positive from equation \eqref{eqn:det} whiles the trace can assume either a positive or a negative number.

\end{document}